\documentclass[aihp,preprint]{imsart}

\RequirePackage[OT1]{fontenc}
\RequirePackage{amsmath,amssymb,amsthm,natbib}
\usepackage{epsfig,color,colordvi}
\definecolor{my-blue}{rgb}{0.0,0.0,0.6}
\definecolor{my-red}{rgb}{0.5,0.0,0.0}
\definecolor{my-green}{rgb}{0.0,0.5,0.0}
\RequirePackage[colorlinks,urlcolor=my-blue,linkcolor=my-red,citecolor=my-green]{hyperref}
\RequirePackage{hypernat}
\usepackage{mathrsfs}

\numberwithin{equation}{section}

\allowdisplaybreaks[1]

\theoremstyle{definition}
\newtheorem{definition}{Definition}[section]

\theoremstyle{remark}
\newtheorem{remark}{Remark}[section]

\theoremstyle{plain}
\newtheorem{theorem}{Theorem}[section]
\newtheorem{lemma}[theorem]{Lemma}

\newcommand*{\beq}{\begin{equation}}
\newcommand*{\eeq}{\end{equation}}
\newcommand*{\bal}{\begin{aligned}}
\newcommand*{\eal}{\end{aligned}}
\newcommand*{\nn}{\nonumber}

\begin{document}
                      
\providecommand{\abs}[1]{\left\vert#1\right\vert}
\newcommand*{\one}{{{\rm 1\mkern-1.5mu}\!{\rm I}}}
\newcommand*{\esssup}{\mathop{{\rm ess~sup}}}
\newcommand*{\e}{\varepsilon}
\newcommand*{\w}{\omega}
\newcommand*{\cA}{{\mathcal A}} 
\newcommand*{\cG}{{\mathcal G}}
\newcommand*{\cI}{{\mathcal I}}
\newcommand*{\cK}{{\mathcal K}}
\newcommand*{\cM}{{\mathcal M}}
\newcommand*{\cN}{{\mathcal N}}
\newcommand*{\cP}{{\mathcal P}}
\newcommand*{\cQ}{{\mathcal Q}}
\newcommand*{\cX}{{\mathcal X}}
\newcommand*{\cY}{{\mathcal Y}}
\newcommand*{\sB}{{\mathscr B}}
\newcommand*{\sC}{{\mathscr C}}
\newcommand*{\sF}{{\mathscr F}}
\newcommand*{\sR}{{\mathscr R}}
\newcommand*{\E}{{\mathbb E}}
\renewcommand*{\P}{{\mathbb P}}
\newcommand*{\N}{{\mathbb N}}
\newcommand*{\R}{{\mathbb R}}
\newcommand*{\fhat}{{\hat f}}
\newcommand*{\qhat}{{\hat q}}
\newcommand*{\zhat}{{\hat z}}
\newcommand*{\muhat}{{\hat\mu}}
\newcommand*{\pihat}{{\hat\pi}}
\newcommand*{\abar}{{\bar a}} 
\newcommand*{\qbar}{{\bar q}}
\newcommand*{\xbar}{{\bar x}}
\newcommand*{\zbar}{{\bar z}}
\newcommand*{\mubar}{{\bar\mu}}
\newcommand*{\ombar}{{\bar\omega}}
\newcommand*{\xtil}{{\tilde x}}
\newcommand*{\ztil}{{\tilde z}}
\newcommand*{\Z}{{\mathbb Z}}
\newcommand*{\kS}{{\mathfrak S}}
\newcommand*{\si}{\sigma}
\def\mydot{{\,{}_{{}^{\scriptstyle\centerdot}}}}
\def\X{\cX}
\def\Y{\cY}
\def\unK{\underline K}
\def\Iq{I_{\rm quen}} 
\def\Ia{I_{\rm avg}}  
\def\Hq{H_{\rm quen}}	
\def\Ha{H_{\rm avg}} 	
\def\bigom{\mathbf\Omega} 
\def\SP{\cS} 
\def\range{{\sR}} 
\def\wz{\eta} 
\def\MC{\cQ} 
\def\measures{\cM_1} 
\def\Hrange{R} 
\def\Hergod{E} 
\def\Hellip{R} 
\def\Hmom{M} 
\def\step{M} 
\def\Sop{S} 
\def\Proj{V} 
\def\Sopr{S^+}
\def\Sopl{S^-}
\def\Soplr{S}
\def\Uop{U}
\def\pr{p^+}
\def\pl{p^-}
\def\quniv{\qbar}
\def\M{\cM} 
\def\K{\cK} 
\def\Ber{B}
\def\rrange{\range} 
\def\W{\mathbf W}

\begin{frontmatter}
\title{Process-level quenched large deviations for random walk in random environment}
\runtitle{Quenched LDP for RWRE}

\begin{aug}
\author{\fnms{Firas} \snm{Rassoul-Agha$^{\rm a,1}$}
\ead[label=e1]{firas@math.utah.edu}
\ead[label=u1,url]{http://www.math.utah.edu/$\sim$firas}}
\and
\author{\fnms{Timo} \snm{Sepp\"al\"ainen$^{\rm b,2}$}\ead[label=e2]{seppalai@math.wisc.edu}\ead[label=u2,url]{http://www.math.wisc.edu/$\sim$seppalai}
}
\runauthor{F. Rassoul-Agha and T. Sepp\"al\"ainen}

\affiliation{University of Utah and University of Wisconsin-Madison}

\address{$^{\rm a}$Department of Mathematics,
University of Utah,
155 South 1400 East,
Salt Lake City, UT 84109,
USA.\\
\printead{e1}}
\address{$^{\rm b}$Department of Mathematics,
University of Wisconsin-Madison,
419 Van Vleck Hall,
Madison, WI 53706,
USA.\\
\printead{e2}}

\end{aug}

\centerline{}
\centerline{\footnotesize Received 4 September 2009; revised 8 April 2010; accepted 16 April 2010}
\footnotetext{$^1$Supported in part by NSF Grant DMS-0747758.}
\footnotetext{$^2$Supported in part by NSF Grant DMS-0701091 and by the Wisconsin Alumni Research Foundation.}

\begin{abstract}
We consider a bounded step size random walk in an ergodic
random environment with some ellipticity, on an integer lattice of arbitrary 
dimension.
We prove  a level 3 large deviation principle, under almost every  environment, 
with rate function related to a relative entropy.
\end{abstract}

\begin{abstract}[language=french]
Nous consid\'erons une marche al\'eatoire en environment 
al\'eatoire ergodique. La marche est elliptique et \`a pas born\'es.
Nous prouvons un principe de grandes d\'eviations au niveau 3, sous presque tout
environnement, avec une fonctionnelle d'action  li\'ee \`a une entropie relative.
\end{abstract}



\end{frontmatter}

\section{Introduction}{\label{intro}
We describe the standard model of random walk in random environment (RWRE) on $\Z^d$.
Let $\Omega$ be a Polish space and $\kS$ its Borel $\si$-algebra. 
Let $\{T_z:{z\in\Z^d}\}$ be a group of 
continuous commuting bijections on $\Omega$: $T_{x+y}=T_xT_y$ and $T_0$ is the identity.
Let $\P$ be a $\{T_z\}$-invariant probability measure on $(\Omega,\kS)$ that is ergodic under this group.
In other words, the $\sigma$-algebra of Borel sets invariant under $\{T_z\}$ is trivial under $\P$.

Denote the space of probability distributions
on $\Z^d$ by 
 $\cP=\{(p_z)_{z\in\Z^d}\in[0,1]^{\Z^d}:\sum_z p_z=1\}$ and
 give it the weak topology or, equivalently, the restriction of the product topology.
  Let $\w\mapsto(p_z(\w))_{z\in\Z^d}$ be a 
continuous mapping from $\Omega$ to $\cP$. For $x,y\in\Z^d$ define $\pi_{x,y}(\w)=p_{y-x}(T_x\w)$.
We call $\w$ and also $(\pi_{x,y}(\w))_{x,y\in\Z^d}$ an environment because it determines the transition 
probabilities of a Markov chain.

The set of admissible steps is denoted by $\range=\{z:\E[\pi_{0,z}]>0\}$.  
One can then redefine  
  $\cP=\{(p_z)_{z\in\range} \in[0,1]^\range:\sum_z p_z=1\}$ and
  transition probabilities $\pi_{x,y}$ are defined only 
  for $x,y\in\Z^d$ such that  $y-x\in\range$.  

Given  $\w$ and a starting point $x\in\Z^d$, let $P_x^\w$ be the law of the Markov chain $X_{0,\infty}=(X_n)_{n\ge0}$ on $\Z^d$, starting at $X_0=x$ and having
transition probabilities $(\pi_{y,y+z}(\w))$. That is,
	\[P_x^\w\{X_{n+1}=y+z\,|\,X_n=y\}=\pi_{y,y+z}(\w),\text{ for all $y,z\in\Z^d.$}\] 
$X_{0,\infty}$ is called a random walk in environment $\w$ and $P_x^\w$ is called the {\sl quenched} distribution.  The {\sl joint} distribution is 
  $P_x(dx_{0,\infty},d\w)=P_x^\w(dx_{0,\infty})\P(d\w)$.  
Its marginal on $(\Z^d)^{\Z_+}$ is also denoted by $P_x$ and   called the {\sl averaged} 
(or {\sl annealed}) 
distribution since $\w$ is averaged out:  
	\[P_x(A)=\int P_x^\w(A)\,\P(d\w) \quad\text{  for a 
measurable $A\subset(\Z^d)^{\Z_+}$.}\]
The canonical case of the above setting is $\Omega=\cP^{\Z^d}$ and $p_z(\w)=(\w_0)_z$. 

Next  a quick description of the problem we are interested in.
Assume given a sequence of probability measures $Q_n$ on a Polish space $(\X,\sB_\X)$
and a lower semicontinuous  function $I:\X\to[0,\infty]$.  Then  
  the {\sl large deviation upper bound} holds with   {\sl rate function} $I$ if
\begin{align*}
	&\varlimsup_{n\to\infty}n^{-1}\log Q_n(C)\le-\inf_C I\text{ for all closed sets }C\subset\X.
\end{align*}
Similarly,    rate function $I$ governs the {\sl large deviation lower bound} if
\begin{align*}
	&\varliminf_{n\to\infty}n^{-1}\log Q_n(O)\ge-\inf_O I\text{ for all open sets }O\subset\X. 
\end{align*}
If both hold with the same rate function $I$, then  the {\sl large deviation
principle} (LDP) holds with rate $I$.   We shall use basic, well known
features  of large deviation theory and relative entropy
without citing every instance.  The reader can consult  references 
\citep{demb-zeit-ldp}, \citep{denholl-ldp},  
\citep{deus-stro-ldp}, \citep{rass-sepp-ldp}, and \citep{vara-ldp}. 

If the   upper bound
(resp.\ lower bound, resp.\ LDP) holds with some function $I:\X\to[0,\infty]$, then it also holds with   the {\sl lower semicontinuous regularization}  $I_{\rm{lsc}}$ 
of $I$ defined by 
			\[I_{\rm{lsc}}(x)=\sup\Big\{\inf_O I:x\in O\text{ and $O$ is open}\Big\}.\]
Thus  the
rate function can  be required to be lower semicontinuous, and then it 
is unique.

Large deviations arrange themselves more or less naturally in three levels.  Most of the work
on quenched large deviations for 
 RWRE has been at 
{\sl level 1}, that is, on large deviations for $P_0^\w\{X_n/n\in\cdot\}$.
 Greven and den Hollander \citep{grev-holl-94} considered the product one-dimensional nearest-neighbor case, 
 Comets, Gantert, and Zeitouni \citep{come-gant-zeit-00} the ergodic one-dimensional nearest-neighbor case,
Yilmaz \citep{yilm-09-b} the ergodic one-dimensional case with bounded step size, 
Zerner \citep{zern-98}  
the multi-dimensional product nestling case, and Varadhan \citep{vara-03}  the general ergodic multidimensional case with
bounded step size.  Rosenbluth \citep{rose-thesis} gave a variational formula for the rate function in \citep{vara-03}.
 {\sl Level 2} quenched large deviations appeared in the work of 
   Yilmaz \citep{yilm-09-b} for the distributions
 $P_0^\w\{n^{-1}\sum_{k=0}^{n-1}\delta_{T_{X_k}\w,Z_{k+1}}\in\cdot\}$. 
Here $Z_k=X_k-X_{k-1}$ denotes the step of the walk.  
 

Our object of study, {\sl level 3} or {\sl process level} large deviations concerns the 
{\sl empirical process}  
 \beq R_n^{1,\infty}=n^{-1}\sum_{k=0}^{n-1}\delta_{T_{X_k}\w,Z_{k+1,\infty}}\label{empproc}\eeq
 where  $Z_{k+1,\infty}=(Z_i)_{i\ge k+1}$ denotes the entire sequence of future steps.  
Quenched distributions
	$P_0^\w\{R_n^{1,\infty}\in \cdot\}$ are   probability measures on the space 
  $\M_1(\Omega\times\range^\N)$.  This is  the space of Borel probability measures on
  $\Omega\times\range^\N$ endowed with the weak topology generated by bounded
  continuous functions.  
  
 The levels do form a hierarchy: higher level LDPs can be projected down to 
 give LDPs at lower levels.  Such results are called  {\sl contraction
principles} in large deviation theory.

The main technical contribution of this work is the  extension 
of a homogenization   argument that proves the upper bound to the multivariate level 2 setting.  
This idea goes back to  Kosygina, Rezakhanlou, and Varadhan \citep{kosy-reza-vara-06} in the context of
diffusions with random drift,   and was   used by both Rosenbluth \citep{rose-thesis} and Yilmaz \citep{yilm-09-b} 
to prove their  LDPs.  

Before turning to specialized assumptions and notation, here are some general 
conventions.     $\Z_+$, $\Z_-$, and $\N$ denote, respectively, the set of non-negative, 
non-positive, and positive integers.
  $|\cdot|$ denotes the $\ell^\infty$-norm on $\R^d$.  $\{e_1,\dotsc,e_d\}$ is the canonical basis of $\R^d$. 
In addition to  $\measures(\X)$ for the space 
  of probability measures on $\X$, we write  $\MC(\X)$ for  the set  of 
Markov transition kernels on $\X$.  Our spaces are   Polish and the $\sigma$-algebras
Borel.  
Given   $\mu\in\M_1(\X)$ and  
$q\in\MC(\X)$,
 $\mu\times q$ is the probability measure on 
$\X\times\X$ defined by 
\[\mu\times q(A\times B)=\int \one_A(x)q(x,B)\,\mu(dx)\]
and $\mu q$ is  its second marginal.
 For a probability measure $P$,
$E^P$ denotes the corresponding expectation operator. Occasionally  $P(f)$  may replace
 $E^P[f]$. 

\section{Main result}
Fix a dimension $d\ge 1$.  
Following are the hypotheses for the level 3 LDP. In Section \ref{prelim}
we refine these to state precisely what is used by different parts of the proof.
\begin{align}\label{RWRE-compact} \text{ $ \range$  is finite and $\Omega$ is a compact metric space.} \end{align}
\begin{align}\label{RWRE-elliptic}
	\text{$ \forall x\in\Z^d$,  $ \exists m\in\N$ and $ z_1,\dotsc,z_m\in\rrange$ such that 
	$x=z_1+\cdots+z_m$.}
\end{align}
%
%
%
%
	\begin{align}\label{RWRE-moment}
	\text{$\exists p>d$ such that $\E[\,\abs{\log\pi_{0,z}}^p\,]<\infty$  $\forall z\in\rrange$.}
	\end{align}

When $\range$ is finite the canonical $\Omega=\cP^{\Z^d}$  is compact.
The commonly used  assumption of 
uniform ellipticity, namely the existence of $\kappa>0$ such that
 $\P\{\pi_{0,z}\ge\kappa\}=1$ for  
$z\in\range$ and $\range$ contains the $2d$ unit vectors,  
  implies  assumptions \eqref{RWRE-elliptic} and  \eqref{RWRE-moment}. 



  We need notational apparatus for backward, forward, and bi-infinite paths.  
The  increments 
 of   a bi-infinite path $(x_i)_{i\in\Z}$ in $\Z^d$  with $x_0=0$
are denoted by $z_i=x_i-x_{i-1}$. The  sequences $(x_i)$ and $(z_i)$ 
are in 1-1 correspondence. Segments of sequences are denoted by     
$z_{i,j}=(z_i,z_{i+1},\cdots,z_j)$,   also for $i=-\infty$ or $j=\infty$, and also 
  for   random variables:   
$Z_{i,j}=(Z_i,Z_{i+1},\cdots,Z_j)$. 
	
In general  $\wz_{i,j}$   denotes the pair $(\w,z_{i,j})$, but when  
  $i$ and $j$ are clear from the context we write simply  $\wz$.
We will also sometimes abbreviate $\wz_-=\wz_{-\infty,0}$. 
The spaces to which elements $\wz$ belong are 
 $\bigom_{-}=\Omega\times\range^{\Z_-}$,
$\bigom_+=\Omega\times\range^\N$ and  $\bigom=\Omega\times\range^\Z$.
Their relevant  shift transformations are   
\begin{align*}
	&\Sopl_z:\,\bigom_-\to\bigom_-:\,(\w,z_{-\infty,0})\mapsto(T_z\w,z_{-\infty,0},z),\\
	&\Sopr:\,\bigom_+\to\bigom_+:\,(\w,z_{1,\infty})\mapsto(T_{z_1}\w,z_{2,\infty}),\\
	&\Soplr:\,\bigom\to\bigom:\,(\w,z_{-\infty,\infty})\mapsto(T_{z_1}\w,\zbar_{-\infty,\infty}),
\end{align*}
where $\zbar_i=z_{i+1}$. We use the same symbols $\Sopl_z$, $\Sopr$, and $\Soplr$ to act on $z_{-\infty,0}$, $z_{1,\infty}$, and $z_{-\infty,\infty}$
in the same way.

The empirical process \eqref{empproc} lives in $\bigom_+$ but the rate function
is best defined in terms of backward paths.  Invariance allows us to pass conveniently
between theses settings.  
If   $\mu\in\M_1(\bigom_+)$ is 
$\Sopr$-invariant, it has 
a unique $\Soplr$-invariant extension 
$\mubar$ on $\bigom$. 
Let $\mu_-=\mubar_{|\bigom_-}$, the restriction of $\mubar$ to its marginal
on $\bigom_-$.    There is a unique kernel $q_\mu$ on $\bigom_-$ that fixes
$\mu_-$ (that is,  $\mu_-q_\mu=\mu_-$)  and satisfies  
\begin{align}\label{q-support}
q_\mu(\wz_-,\{\Sopl_z\wz_-\,:\,z\in\range\})=1\text{ for $\mu_-$-a.e.\ $\wz_-$.}
\end{align}
Namely 
 	\begin{align*}
		q_\mu(\wz_-,\Sopl_z\wz_-)=\mubar\{Z_1=z\,|\,(\w,Z_{-\infty,0})=\wz_-\}.  
	\end{align*}
(Uniqueness here is $\mu_-$-a.s.)   
Indeed, on the one hand, the above $q_\mu$ does leave $\mu_-$ invariant. 
On the other hand, if $q$ is a kernel supported on shifts and leaves $\mu_-$ invariant, 
and if $f$ is a bounded measurable
function on $\bigom_-$, then 
	\begin{align*}
		\int q(\wz_-,\Sopl_z\wz_-)f(\wz_-)\mubar(d\wz)
		&=\sum_{z'}\int q(\wz_-,\Sopl_{z'}\wz_-)f(T_{z'-z}\w,z_{-\infty,0})\one\{z'=z\}\mu_-(d\wz_-)\\
		&=\int f(T_{-z}\w,z_{-\infty,-1})\one\{z_0=z\}\mu_-(d\wz_-)\\
		&=\int f(\wz_-)\one\{z_1=z\}\mubar(d\wz)\\
		&=\int q_\mu(\wz_-,\Sopl_z\wz_-)f(\wz_-)\mubar(d\wz).
	\end{align*}
	
The RWRE transition gives us  the kernel $\pl\in\MC(\bigom_-)$ defined by
	\begin{align*}
		&\pl(\wz_-,\Sopl_z\wz_-)=\pi_{0,z}(\w),\text{ for }\wz_-=(\w,z_{-\infty,0})\in\bigom_-.  
	\end{align*}
If $q\in\MC(\bigom_-)$  
satisfies 
$\mu_-\times q\ll\mu_-\times \pl$, then $q(\wz_-,\{\Sopl_z\wz_-:z\in\range\})=1$ $\mu_-$-a.s.
and their relative entropy  is given by
\beq\begin{aligned}
	H(\mu_-\times q\,|\,\mu_-\times\pl)
	=\int\sum_{z\in\range}q(\wz_-,\Sopl_z\wz_-)\,\log\frac{q(\wz_-,\Sopl_z\wz_-)}{\pl(\wz_-,\Sopl_z\wz_-)}\,\mu_-(d\wz_-).
\end{aligned}\label{entr1}\eeq

Let $\mu_0$ denote the marginal of $\mu$ on $\Omega$.
Our main theorem is the following.

\begin{theorem}\label{main}   
	Let $(\Omega,\kS,\P,\{T_z\})$ be an ergodic system.
	Assume  \eqref{RWRE-compact},
	 \eqref{RWRE-elliptic}, and   \eqref{RWRE-moment}.
	 Then, for $\P$-a.e.\ $\w$, the large deviation principle holds for the laws $P_0^\w\{R_n^{1,\infty}\in\cdot\}$,
	with rate function  $\Hq:\measures(\bigom_+)\to[0,\infty]$ equal to the lower semicontinuous 
	regularization of the convex function 
		\begin{align}\label{Hdef}
			H(\mu)=
			\begin{cases}
				H(\mu_-\times q_\mu\,|\,\mu_-\times \pl)&\text{if $\mu$ is $\Sopr$-invariant and $\mu_0\ll\P$},\\
				\infty&\text{otherwise.}
			\end{cases}
		\end{align}
\end{theorem}

We make next some observations about the rate function $\Hq$.

\begin{remark}  
As is often the case for process level LDPs,  the rate function is affine. 
This follows because we can replace 
  $q_\mu$ with a  ``universal'' kernel $\quniv$  whose definition is independent of $\mu$.
Namely, define
	\[\Uop:\bigom_-\to\bigom_-:(\w,z_{-\infty,0})\mapsto(T_{-z_0}\w,z_{-\infty,-1}).\]
Then, on the event  where $\lim_{n\to\infty}\frac1n\sum_{k=0}^{n-1}\delta_{\Uop{}^k\wz_-}$ exists  define
	\begin{align}
	\label{kernel}
		\quniv(\wz_-,\Sopl_z\wz_-)=q_\mu(\wz_-,\Sopl_z\wz_-)\text{ for }\mu=\lim_{n\to\infty}\frac1n\sum_{k=0}^{n-1}\delta_{U^k\wz_-}.
	\end{align}
On the complement, set $\quniv(\wz_-,\Sopl_z\wz_-)=\delta_{z_0}(z)$, for some fixed $z_0\in\range$.
\end{remark}

\begin{remark}
Let us also  recall the convex analytic characterization of  l.s.c.\ regularization. 
 Let $\sC_b(\X)$ denote the space of bounded continuous functions on $\X$.
Given a function $J:\measures(\X)\to[0,\infty]$, let $J^*:\sC_b(\X)\to\R$ be its {\sl convex conjugate} defined by
	\[J^*(f)=\sup_{\mu\in\measures(\X)}\{E^\mu[f]-J(\mu)\}\]
and let $J^{**}:\measures(\X)\to\R$ be its {\sl convex biconjugate} defined by
	\[J^{**}(\mu)=\sup_{f\in\sC_b(\X)}\{E^\mu[f]-J^*(f)\}.\]
If $J$ is convex and not identically infinite, $J^{**}$ is the same as its lower semicontinuous regularization $J_{\rm{lsc}}$;
see Propositions 3.3 and 4.1 of \citep{ekel-tema-cvxanal} or Theorem 5.18 of \citep{rass-sepp-ldp}.
Thus the rate function in Theorem \ref{main} is  $\Hq=H^{**}$. 
\label{lscremark} \end{remark}


As expected, rate function $H$ has in fact an alternative representation as a specific relative entropy.
For a probability measure $\nu$ on $\Omega$, define the probability measure $\nu\times P^\mydot_0$ on $\bigom_+$ by 
\[  \int_{\bigom_+}\!\!\!\! f \,d(\nu\times P^\mydot_0) = \int_\Omega 
\Big[\int_{\range^\N} f(\w,z_{1,\infty})\,P^\w_0(dz_{1,\infty} )\Big]\,\nu(d\w).  \]
On any of the product spaces of environments and paths, define the $\sigma$-algebras
$\cG_{m,n}=\sigma\{\w, z_{m,n}\}$.    Let $H_{\cG_{m,n}}(\alpha\,|\,\beta)$ denote 
the relative entropy of the restrictions of the probability measures
$\alpha$ and $\beta$ to the $\sigma$-algebra 
$\cG_{m,n}$.  Let $\Pi$ be the kernel of the environment chain $(T_{X_n}\w)$,
defined as $\Pi f(\w)=E_0^\w[f(T_{X_1}\w)]=\sum_z \pi_{0,z}(\w) f(T_z\w)$.

\begin{lemma}  Let $\mu\in\cM_1(\mathbf{\Omega}_+)$ be $\Sopr$-invariant.    Then the limit
\begin{align}
h(\mu\,|\,\mu_0\times P^\mydot_0) =
 \lim_{n\to\infty}\frac1n  H_{\cG_{1,n}}(\mu\,|\,\mu_0\times P^\mydot_0)
 \label{specentr1}
 \end{align}
exists and equals $H(\mu_-\times q_\mu\,|\,\mu_-\times\pl)$. 
 \end{lemma}
 
 \begin{proof}
Fix $\mu$.   Let $\mu_i^{\w, z_{1,i-1}}(\cdot)$ denote the conditional distribution 
 of $Z_i$ under $\mu$, given $\cG_{1,i-1}$.  Then by the $\Sop$-invariance, 
\[   {\bar\mu}[ Z_1=u\,\vert\,\cG_{2-i,0}](\w,z_{2-i,0}) =  \mu_i^{T_{x_{1-i}}\w, z_{2-i,0}}(u).  \]
For $i=1$ we must interpret $\cG_{1,0}=\sigma\{\w\}=\mathfrak S$ and 
$(\w, z_{1,0})$  simply as $\w$. 
Observe also that the conditional distribution 
 of $Z_i$ under $\mu_0\times P^\mydot_0$, given $\cG_{1,i-1}$, is  
 $\pi_{0,\centerdot}(T_{x_{i-1}}\w)$.  

  By two applications of  the conditional entropy formula (Lemma 10.3 of \citep{vara-ldp} or Exercise 6.14 of \citep{rass-sepp-ldp}),   
\begin{align}
\begin{split}
&H_{\cG_{1,n}}(\mu\,|\,\mu_0\times P^\mydot_0)
= \sum_{i=1}^n \int H\bigl( \mu_i^{\w, z_{1,i-1}}\,\big| \,\pi_{0,\centerdot}(T_{x_{i-1}}\w)\bigr)\,
\mu(d\w, dz_{1,\infty}) \\
&\quad = 
\sum_{i=1}^n \int H\bigl( \,  {\bar\mu}[ Z_1=\cdot \,|\,\cG_{2-i,0}](T_{x_{i-1}}\w,z_{1,i-1})  
\,\big\vert \,\pi_{0,\centerdot}(T_{x_{i-1}}\w)\bigr)\,
\mu(d\w, dz_{1,\infty}) \\
&\quad = 
\sum_{i=1}^n \int H\bigl( \,  {\bar\mu}[ Z_1=\cdot \,|\,\cG_{2-i,0}](\w,z_{2-i,0})\,\big| \,\pi_{0,\centerdot}(\w)\bigr)\,
\mu_-(d\w, dz_{-\infty, 0}) \\
&\quad = \sum_{i=1}^n   H_{\cG_{2-i,1}}(\mu_-\times q_\mu\,|\, \mu_-\times \pl). 
\end{split}\label{entraux1}
\end{align}
As $k\to\infty$,  the $\sigma$-algebras $\cG_{-k,1}$ generate the 
$\sigma$-algebra $\cG_{-\infty,1}=\sigma\{\w, z_{-\infty, 1}\}$, and consequently 
\begin{align}
H_{\cG_{2-i,1}}(\mu_-\times q_\mu\,|\, \mu_-\times \pl) 
\nearrow   H(\mu_-\times q_\mu\,|\, \mu_-\times \pl)
\qquad\text{as $i\nearrow\infty$}.
\label{increasing entropy}
\end{align} 

We have taken some liberties with notation and regarded  $\mu_-\times q_\mu$ and  $\mu_-\times \pl$ as measures on the variables $(\w, z_{-\infty, 1})$, instead of on pairs
$((\w, z_{-\infty, 0}), (\w', z_{-\infty, 0}'))$.   This is legitimate because the simple 
structure of
the kernels $q_\mu$ and $\pl$, namely \eqref{q-support}
 implies that $z_{-\infty, 0}'=z_{-\infty, 1}$ and  $\w'=T_{z_1}\w$ almost surely under these
 measures. 

The claim follows by dividing through \eqref{entraux1}
by $n$ and letting $n\to\infty$.  
 \end{proof}

Note that the specific entropy in \eqref{specentr1} is not an entropy between two
$\Sopr$-invariant measures unless $\mu_0$ is $\Pi$-invariant.
The next lemma exploits the previous one to say something about the zeros of $\Hq$.

\begin{lemma}  If $\Hq(\mu)=0$ then  
$\mu(d\w, dz_{1,\infty})=\mu_0(d\w)P^\w_0(dz_{1,\infty})$ for some $\Pi$-invariant $\mu_0$.
\end{lemma}

Note that it is not necessarily true that $\mu_0\ll\P$ in the above lemma.

 \begin{remark}
 One can show that under \eqref{RWRE-compact} and \eqref{RWRE-elliptic} there is at most one $\P_\infty\in\M_1(\Omega)$ that is $\Pi$-invariant 
 and such that $\P_\infty\ll\P$; see for example \citep{rass-03}. In fact, in this case $\P_\infty\sim\P$.
 The above lemma shows that the zeros of $\Hq$ consist of $P_0^\infty=\P_\infty\times P_0^\mydot$ (if $\P_\infty\ll\P$ exists) and possibly measures of the
 form $\mu_0\times P_0^\mydot$, with $\mu_0$ being $\Pi$-invariant but such that $\mu_0\not\ll\P$. 
 \end{remark}
 
\begin{proof}
  There is a sequence of 
$\Sopr$-invariant probability measures $\mu^{(m)}\to\mu$ such that 
$H(\mu^{(m)})\to 0$ and $\mu^{(m)}_0\ll\P$. 
(If $\mu_0\ll\P$ then we can take  $\mu^{(m)}=\mu$.)    Let $\mu^{(m)}_1$ denote the
marginal distribution on $(\w,z_{-\infty, 1})$ which can be identified with 
$\mu^{(m)}_-\times q_{\mu^{(m)}}$ and converges to the corresponding marginal $\mu_1$.   
 By the continuity of
the kernel $\pi_{0,z}(\w)$,  $\mu^{(m)}_-\times \pl\to\mu_-\times \pl$. 
From these limits and the lower semicontinuity of relative entropy, 
\[
H(\mu_1\,\vert\, \mu_-\times \pl)=\lim_{m\to\infty} H(\mu^{(m)}_1\,\vert\, \mu^{(m)}_-\times \pl)
= 0.   \]
This tells us that $\mu_-$ is $\pl$-invariant, which in turn implies that $\mu_0$ is $\Pi$-invariant,
and  together with the $\Sopr$-invariance of $\mu$ implies also that $\mu=\mu_0\times P^\mydot_0$.  (The last point can also be seen 
from \eqref{entraux1} and \eqref{increasing entropy}.)
\end{proof}

We close this section with some examples. 

Let $\Omega=\cP^{\Z^d}$ with $\cP=\{(p_z)_{z\in\range}\in(0,1)^\range:\sum_z p_z=1\}$.
Let $\nu\in\bigom_+$ be the law of a classical random walk; i.e.\ $\nu=\nu_0\times P^\mydot_0$ with $\nu_0=\delta_\alpha^{\otimes\Z^d}$,
for  some $\alpha\in\cP$. Then $H(\nu_-\times q_\nu\,|\,\nu_-\times\pl)=0$. 
However, if  $\sum_z z \alpha_z$ is not in the set $\cN=\{E^\mu[Z_1]:\Hq(\mu)=0\}$, then, $\Hq(\nu)>0$. Note that by 
the contraction principle, $\cN$ is the zero set of the level-1 rate function. Hence if $\P$ is product, by \citep{vara-03} $\cN$ consists
of a singleton or a line segment. 
Thus we can pick $\alpha$ so that the mean $\sum z\alpha_z$
does not lie in $\cN$, and consequently we have measures $\nu$ for which
$\Hq(\nu)>0=H(\nu_-\times q_\nu\,|\,\nu_-\times\pl)$.
That is, the rate $\Hq$ does not have to pick up the entropy value.

Lower semicontinuity of relative entropy implies $\Hq(\mu)=H(\mu)$ when $\mu_0\ll\P$. 
This equality can still happen when $\mu_0\not\ll\P$; i.e.\ the 
l.s.c.\ regularization can bring the rate $\Hq$ down from infinity all the way to the entropy. Here is a somewhat singular example. 
Assume $\P\{\pi_{0,0}(\w)>0\}=1$ and  let $\zeta=(0,0,0,\dotsc)$ be the constant sequence
of $0$-steps in $\Z^d$.   For each $\ombar\in\Omega$ define the 
(trivially $\Sopr$-invariant)  probability measure
$\nu^\ombar=\delta_{(\ombar,\zeta)}$ on $\bigom_+$.   Then, for all
$\ombar$ in the (minimal closed) support of $\P$,
\begin{align}  
\Hq(\nu^\ombar) =H(\nu^{\ombar}_-\times q_{\nu^{\ombar}_-} \,\vert\,\nu^{\ombar}_-\times \pl) =-\log \pi_{0,0}(\ombar). 
 \label{exline2}
 \end{align}
The case $ \pi_{0,0}(\ombar)=0$ is allowed here, which can of course happen if uniform
ellipticity is not assumed.  
 
 The second equality in \eqref{exline2} is clear from definitions, because the kernel is trivial:  $q_{\nu^{\ombar}_-}(\wz_-, \wz_-)=1$. 
Since $\Hq(\nu^\ombar)$ 
is defined by l.s.c.\ regularization and 
 entropy itself is l.s.c.,  entropy always gives a lower bound for $\Hq$.  
 If $\P\{\ombar\}>0$ then $\nu^\ombar_0=\delta_\ombar\ll\P$
and   the first equality in \eqref{exline2} is true by definition.  
If $\P\{\ombar\}=0$ pick a sequence of open neighborhoods $G_j\searrow\ombar$.
The assumption 
that $\ombar$ lies in the support of $\P$ implies $\P(G_j)>0$.  Define a sequence
of approximating measures by 
$\mu^j=\frac1{\P(G_j)} \int_{G_j} \nu^\w \,\P(d\w)$ 
with entropies 
\[   H(\mu^j\times q_{\mu^j}\, \vert\, \mu^j\times \pl) 
= - \frac1{\P(G_j)} \int_{G_j} \log \pi_{0,0}(\w)  \,\P(d\w). \]
The above entropies converge 
to $-\log \pi_{0,0}(\ombar)$ by continuity of $\pi_{0,0}(\cdot)$.  We have 
verified \eqref{exline2}. 




\section{Multivariate  level 2 and setting the stage for the proofs}\label{prelim}

The assumptions made for the main result are the union of all the assumptions used
in this paper.   To facilitate future work, we next list the different assumption that
are needed for different parts of the proof.  

The lower bounds in Theorem \ref{main} above and Theorem \ref{MC-qldp-th} below do not require $\Omega$  compact nor
$\range$ finite. They hold under the assumption
that $\P$ is ergodic for $\{T_z:z\in\range\}$ and the following two conditions are satisfied.
	\beq  \text{ $\P\{\pi_{0,z}>0\}\in\{0,1\}$ for all $z\in\Z^d$. } \label{RWRE-regular}\eeq
	\begin{align}\label{RWRE-invariant}
	\begin{split}
	&\text{Either $\E[|\log\pi_{0,z}|]<\infty$ holds for all $z\in\rrange$ or there exists a}\\ 
	&\text{probability measure $\P_\infty$ on $(\Omega,\kS)$ with $\P_\infty\Pi=\P_\infty$ and $\P_\infty\ll\P$.}
	\end{split}
	\end{align}
Note that \eqref{RWRE-regular} is
a regularity condition that says that either all environments allow the move or all prohibit it. 

Our proof of the upper bound uses stricter assumptions. The upper bound
holds if $\P$ is ergodic for $\{T_z:z\in\range\}$, $\range$ is finite, $\Omega$ is compact,
the moment assumption \eqref{RWRE-moment} holds, and 
	\begin{align}\label{RWRE-elliptic2}
	\text{$ \forall x\in\range$,  $ \exists m\in\N$, $\exists z_1,\dotsc,z_m\in\rrange$ such that 
	$x+z_1+\cdots+z_m=0$.}
	\end{align}
On its own, \eqref{RWRE-elliptic2} is weaker than \eqref{RWRE-elliptic}. However, since the additive group generated by $\range$
is isomorphic to $\Z^{d'}$ for some $d'\le d$, we always assume, without any loss of generality, that
	\begin{align}\label{RWRE-span}
	\text{$\Z^d$ is the smallest additive group containing $\range$.}
	\end{align}
Then, under \eqref{RWRE-span}, \eqref{RWRE-elliptic2} is equivalent to \eqref{RWRE-elliptic}.


The only place where the condition $p>d$ (in \eqref{RWRE-moment}) is needed is for Lemma \ref{CLASS K} to hold. See
Remark \ref{explanation}.
The only place where \eqref{RWRE-elliptic} (or \eqref{RWRE-elliptic2}) is needed is in the proof of \eqref{unif-Lp} in Lemma \ref{F-lemma}.
This is  the only reason that our result does not cover the so-called {\sl forbidden direction} case. A particularly interesting special case is the 
space-time, or dynamic, environment; i.e.\ when $\range\subset\{z:z\cdot e_1=1\}$. 
A level 1 quenched LDP can be proved for space-time RWRE through the subadditive ergodic
theorem, as was done for elliptic walks in \citep{vara-03}.   
Yilmaz \citep{yilm-09-a}  has shown that for i.i.d.\ space-time RWRE in 
4 and higher dimensions  the quenched
and averaged level 1 rate functions coincide in a neighborhood of the limit velocity. 
In contrast with large
deviations,  the functional central limit theorem of i.i.d.\ space-time RWRE is completely understood; see \citep{rass-sepp-05},
and also \citep{bold-minl-pell-04} for a different proof for steps that have exponential tails. 
  
 Next we turn to the strategy of the proof of Theorem \ref{main}.  
The   process level LDP comes by the familiar projective limit
argument
 from large deviation theory.  The intermediate steps are   
multivariate quenched level 2  LDPs.   
For each $\ell\in\N$ 
define  the multivariate empirical measure 
	\[R_n^{1,\ell}=n^{-1}\sum_{k=0}^{n-1}\delta_{T_{X_k}\w,Z_{k+1,k+\ell}}.\]
This empirical measure lives on 
 the space 
 $\bigom_\ell=\Omega\times\range^\ell$  whose generic element is now denoted by 
$\wz=(\w,z_{1,\ell})$.   

 We can treat $R_n^{1,\ell}$   as the position  level (level 2) empirical measure 
  of a Feller-continuous Markov chain.  
Denote by $P_\wz$  (with expectation $E_\wz$)  the law of the Markov chain $(\wz_k)_{k\ge0}$ on
$\bigom_\ell $ 
with initial state $\wz$ and transition kernel
	\begin{align*}
		&\pr(\wz,\Sopr_z\wz)=\pi_{x_\ell,x_\ell+z}(\w)=\pi_{0,z}(T_{x_\ell}\w),\text{ for }\wz=(\w,z_{1,\ell})\in\bigom_\ell,
	\end{align*}
where
	\[\Sopr_z:\,\bigom_\ell\to\bigom_\ell:\,(\w,z_{1,\ell})\mapsto(T_{z_1}\w,z_{2,\ell},z).\]
  This Markov chain has 
empirical measure
	\[L_n=n^{-1}\sum_{k=0}^{n-1}\delta_{\wz_k}\]
	that satisfies the following LDP.   Define an entropy $H_\ell$ on $\measures(\bigom_\ell)$
	by
	\begin{align}\label{Helldef}
		H_\ell(\mu)=
			\begin{cases}
				\inf\{H(\mu\times q\,|\,\mu\times \pr):q\in\MC(\bigom_\ell)\text{ with }\mu q=\mu\}&\text{if }\mu_0\ll\P,\\
				\infty&\text{otherwise.}
			\end{cases}
	\end{align}
$H_\ell$ is convex by an argument used below at the end
of Section \ref{lower-qldp-section}. Recall Remark \ref{lscremark} about l.s.c.\ regularization.  
 
\begin{theorem}
\label{MC-qldp-th}
	Same assumptions as in Theorem \ref{main}. For any fixed $\ell\ge1$,
	for $\P$-a.e.\ $\w$, and for all $z_{1,\ell}\in\range^\ell$,
	the large deviation principle holds for the sequence of probability measures
	$P_\wz\{L_n\in\cdot\}$ on $\measures(\bigom_\ell)$ with convex rate function $H_\ell^{**}$.
	\end{theorem}


The lower bound in Theorem \ref{MC-qldp-th} follows from a  change of measure and 
the ergodic theorem,  and hints at   the correct rate function.  
  Donsker and Varadhan's \citep{dons-vara-75-a} general Markov chain
   argument gives the upper bound 
but without the absolute continuity restriction
in \eqref{Helldef}.
Thus the main issue is to deal with the case when the rate is infinite. 
This is  nontrivial because  the set of measures with $\mu_0\not\ll\P$ is dense in the set of
probability measures with the same support as $\P$. 
This is where the homogenization argument 
from \citep{kosy-reza-vara-06}, \citep{rose-thesis} and \citep{yilm-09-b}  comes in. 
 
	We conclude this section with a lemma that contains the 
  projective limit step. 

\begin{lemma}
\label{transfer of qldp}
	Assume $\P\in\M_1(\Omega)$ is invariant for the shifts $\{T_z:{z\in\range}\}$ and satisfies the regularity
	assumption \eqref{RWRE-regular}.
	Assume that for each fixed $\ell\ge1$ there exists a rate function
	$I_\ell:\measures(\bigom_\ell)\to[0,\infty]$
	that governs the  large deviation lower bound 
	for the laws $P_\wz\{L_n\in\cdot\}$, for $\P$-almost-every $\w$ and all $z_{1,\ell}\in\bigom_\ell$.
	Then, for $\P$-a.e.\ $\w$, 
	the large deviation lower bound holds for $P_0^\w\{R_n^{1,\infty}\in\cdot\}$ with rate function
		$I(\mu)=\sup_{\ell\ge1}I_\ell(\mu_{|\bigom_\ell})$, for $\mu\in\M_1(\bigom_+).$
		 
	When $\range$ is finite and $\Omega$ is compact the same statement holds for the upper bound 
	and the large deviation principle.  
\end{lemma}

\begin{proof}
	Observe first that $P_\wz$ is the law of $(T_{X_{k}}\w,Z_{k+1,k+\ell})_{k\ge 0}$ under 
	$P_0^\w$, conditioned on $Z_{1,\ell}=z_{1,\ell}$. Since $P_0^\w\{Z_{1,\ell}=z_{1,\ell}\}>0$ $\P$-a.s.\ we have
	for all open sets $O\subset\M_1(\bigom_\ell)$,
		\begin{align*}
			\varliminf_{n\to\infty}n^{-1}\log P_0^\w\{R_n^{1,\ell}\in O\}
			&\ge\varliminf_{n\to\infty}
				n^{-1}\log\Big[P_0^\w\{Z_{1,\ell}=z_{1,\ell}\}P_0^\w\{R_n^{1,\ell}\in O\,|\,Z_{1,\ell}=z_{1,\ell}\}\Big]\\
			&=\varliminf_{n\to\infty}
				n^{-1}\log P_0^\w\{R_n^{1,\ell}\in O\,|\,Z_{1,\ell}=z_{1,\ell}\}
			\ge-\inf_O I_\ell.
		\end{align*}
	
	Similarly, in the case of the upper bound, and when $\range$ is finite, 
	we have for all closed sets $C\subset\M_1(\bigom_\ell)$, 
		\begin{align*}
			\varlimsup_{n\to\infty}n^{-1}\log P_0^\w\{R_n^{1,\ell}\in C\}
			&\le\varlimsup_{n\to\infty}\max_{z_{1,\ell}\in\range^\ell} n^{-1}\log
				P_0^\w\{R_n^{1,\ell}\in C\,|\,Z_{1,\ell}=z_{1,\ell}\}\\
			&\le\max_{z_{1,\ell}\in\range^\ell} \varlimsup_{n\to\infty}
				n^{-1}\log P_0^\w\{R_n^{1,\ell}\in C\,|\,Z_{1,\ell}=z_{1,\ell}\}\\
			&\le-\inf_C I_\ell.
		\end{align*}	

	We conclude that conditioning is immaterial and, $\P$-a.s., the laws of $R_n^{1,\ell}$ induced by $P_0^\w$ satisfy a large deviation lower
	(resp.\ upper) bound governed by $I_\ell$. 	
	The lemma now follows from the Dawson-G\"artner projective limit theorem 
	(see Theorem 4.6.1 in \citep{demb-zeit-ldp}). 
\end{proof}

The next two sections prove
 Theorem \ref{MC-qldp-th}:    lower bound  in Section \ref{lower-qldp-section} and 
  upper bound   in Section \ref{upper-qldp-section}. 
   Section \ref{proof-section}   finishes the proof of 
the main theorem \ref{main}.

\section{Lower bound}
\label{lower-qldp-section}
We now prove the large deviation lower bound in Theorem \ref{MC-qldp-th}.
This section is valid for a general $\range$ that can be infinite and a general Polish $\Omega$.
Lemmas \ref{ergodicity} and \ref{weak-lower-bound} are valid under \eqref{RWRE-regular} only while the lower 
bound proof also requires \eqref{RWRE-invariant}.
Recall that assumption \eqref{RWRE-span} entails no loss of generality.

We start with some ergodicity properties of the measures involved in the definition of
the  function $H_\ell$. Recall that $\bigom_\ell=\Omega\times\range^\ell$ and that for a measure
$\mu\in\measures(\bigom_\ell)$, $\mu_0$ is its marginal on $\Omega$. Denote by $P_0^{(\ell)}$ the law of 
$(\w,Z_{1,\ell})$ under $P_0$.

\begin{lemma}
\label{ergodicity}
	Let $(\Omega,\kS,\P,\{T_z\})$ be ergodic and assume  \eqref{RWRE-regular} and \eqref{RWRE-span} hold.
	Fix $\ell\ge1$ and 
	let $\mu\in\measures(\bigom_\ell)$ be such that $\mu\ll P_0^{(\ell)}$.
	Let $q$ be a Markov transition kernel on $\bigom_\ell$ such that 
		\begin{itemize}
			\item[{\rm(a)}] $\mu$ is $q$-invariant {\rm(}i.e.\ $\mu q=\mu${\rm)};
			\item[{\rm(b)}] $q(\wz,\Sopr_z\wz)>0$ for all $z\in\range$ 
					and $\mu$-a.e.\ $\wz\in\bigom_\ell$;
			\item[{\rm(c)}] $\sum_{z\in\range}q(\wz,\Sopr_z\wz)=1$, for $\mu$-a.e.\ $\wz\in\bigom_\ell$.
		\end{itemize}
	Then, $\mu\sim P_0^{(\ell)}$ and
	the Markov chain $(\wz_k)_{k\ge0}$ on $\bigom_\ell$ with kernel $q$ and 
	initial distribution $\mu$ is ergodic. In particular, we have for all $F\in L^1(\mu)$
	\begin{align}\label{as-erg-thm}
	\lim_{n\to\infty} n^{-1}\sum_{k=0}^{n-1} E^{Q_\wz}[F(\wz_k)]=E^\mu[F],\ \text{for $\mu$-a.e.}\ \wz.
	\end{align}
	Here, $Q_\wz$ is the Markov chain with transition kernel $q$ and initial state $\wz$.
\end{lemma}

\begin{proof}
	First, let us prove mutual absolute continuity. Let $f=\frac{d\mu}{dP_0^{(\ell)}}$. Then, by
	assumptions (a) and (c),
		\begin{align*}
			0=\int\one\{f=0\}f\,dP_0^{(\ell)}=\int\one\{f=0\}\,d\mu
			=\sum_{z\in\range}\int q(\wz,\Sopr_z\wz)\one\{f(\Sopr_z\wz)=0\}\,\mu(d\wz).
		\end{align*}
	By assumption (b), this implies that  for $z\in\range$
		\begin{align*}
			0&=\int \one\{f(\Sopr_z\wz)=0\}\,\mu(d\wz)
			=\int \one\{f(\Sopr_z\wz)=0\}f(\wz)\,P_0^{(\ell)}(d\wz).
		\end{align*}
	By regularity \eqref{RWRE-regular} we conclude that
	$\one\{f(\wz)>0\}\le\one\{f(\Sopr_z\wz)>0\}$, for all $z\in\range$, $z_{1,\ell}\in\range^\ell$, and $\P$-a.e.\ $\w$.

	By first following the path $z_{1,\ell}$, then taking an increment of $z\in\range$, then following a path 
	$\ztil_{1,\ell}\in\range^\ell$, one sees that 
	for all $z_{1,\ell},\ztil_{1,\ell}\in\range^\ell$, all $z\in\range$, and $\P$-a.e.\ $\w$,
			\begin{align}\label{ergodicity-argument}
			\one\{f(\w,z_{1,\ell})>0\}\le\one\{f(T_{x_\ell+z}\w,\ztil_{1,\ell})>0\}.
			\end{align}
	
	Now pick a finite subset $\zhat_1,\dotsc,\zhat_M \in \range$ that generates $\Z^d$ as an additive group;
	e.g.\ take the elements needed for generating the canonical basis $e_1,\dotsc,e_d$.
	Note that $M>d$ can happen; e.g.\ take $d=1$ and $\range=\{2,5\}$.
	
	Applying \eqref{ergodicity-argument} repeatedly, one  can arrange for $z$ to be any point of the form
	$\sum_{i=1}^M k_i \zhat_i$ with  $k_i\in\Z_+$.  
	Furthermore, the ergodicity of $\P$ under shifts $\{T_z\}$ implies its ergodicity under shifts $\{T_{\zhat_1},\dotsc,T_{\zhat_M}\}$,
	since the latter generate the former. 
	We can thus average over $k=(k_1,...,k_M)\in [0,n]^M$, take $n\to\infty$, and
	invoke the multidimensional ergodic theorem (see for example Appendix 14.A of \citep{geor}).
	This shows that 
	for all $z_{1,\ell},\ztil_{1,\ell}\in\range^\ell$ and $\P$-a.e.\ $\w$
			\[\one\{f(\w,z_{1,\ell})>0\}\le \P\{\w:f(\w,\ztil_{1,\ell})>0\}.\]
%
%
%
%
%
%
%
%
%
	Since $f$ integrates to 1 there exists a $z_{1,\ell}\in\range^\ell$ with $\P\{f(\w,z_{1,\ell})>0\}>0$. This implies that 
	$\P\{f(\w,\ztil_{1,\ell})>0\}=1$ for all $\ztil_{1,\ell}\in\range^\ell$ and hence $\mu\sim P_0^{(\ell)}$.
 
	Next, we address the ergodicity issue. 
	By Corollary 2 of Section IV.2 of \citep{rose}, we have that for any $F\in L^1(\mu)$ and $\mu$-a.e.\ $\wz\in\bigom_\ell$,
		\[\lim_{n\to\infty}\frac1n\sum_{k=0}^{n-1}E^{Q_\wz}[F(\wz_k)]=E^\mu[F\,|\,\cI_{\mu,q}].\]
	Here, $\cI_{\mu,q}$ is the $\sigma$-algebra of {\sl $q$-invariant sets}:
		\[\Big\{A\text{ measurable}:\int q(\wz,A)\one_{A^c}(\wz)\,\mu(d\wz)=\int q(\wz,A^c)\one_A(\wz)\,\mu(d\wz)=0\Big\}.\]
	Ergodicity would thus follow from showing that $\cI_{\mu,q}$ is $\mu$-trivial. To this end, let $A$ be $\cI_{\mu,q}$-measurable.
	By assumptions (b) and (c) and mutual absolute continuity we have that for all $z\in\range$
		\[\int\one_A(\Sopr_z\wz)\one_{A^c}(\wz)P_0^{(\ell)}(d\wz)=0.\]
	Replacing the set $\{f>0\}$ by $A^c$, in the above proof of mutual absolute continuity, one concludes that $P_0^{(\ell)}(A)\in\{0,1\}$.
	The same holds under $\mu$ and the lemma is proved.
\end{proof}

We are now ready to derive the lower bound.
We first prove a slightly weaker version. 

\begin{lemma}\label{weak-lower-bound}
	Let $(\Omega,\kS,\P,\{T_z\})$ be ergodic and assume \eqref{RWRE-regular} and \eqref{RWRE-span} hold.
	Fix $\ell\ge1$.
	Then, for $\P$-a.e.\ $\w$, for all $z_{1,\ell}\in\range^\ell$, and for any
	open set $O\subset\measures(\bigom_\ell)$
	\begin{align*}
		&\varliminf_{n\to\infty} n^{-1}\log P_\wz\{L_n\in O\}\\
		&\quad\ge
		-\inf\Big\{H(\mu\times q\,|\,\mu\times\pr):
		\mu\in O,\ \mu_0\ll\P,\ q\in\MC(\bigom_\ell),\ \mu q=\mu,\\
		&\qquad\qquad\qquad\qquad\qquad\qquad\qquad\ 
		\text{ and }\forall z\in\range,\,q(\wz,\Sopr_z\wz)>0,\,\mu\text{-a.s.}\Big\}.
	\end{align*}
\end{lemma}

\begin{proof}
	Fix $\mu\in O$ and $q$ as in the above display. 
	We can also assume that $H(\mu\times q\,|\,\mu\times \pr)<\infty$. Then $q(\wz,\{\Sopr_z\wz:z\in\range\})=1\ \mu$-a.s.
	We can find a weak neighborhood such that $\mu\in B\subset O$. That is, we can find $\e>0$, a positive integer $m$,
	and bounded continuous functions  $f_k:\bigom_\ell\to\R$, 
	such that
		\[B=\{\nu\in\measures(\bigom_\ell):\forall k\le m,\,|E^\nu[f_k]-E^\mu[f_k]|<\e\}.\]
	Let $\sF_n$ be the $\si$-algebra generated by $\wz_0,\dotsc,\wz_n$. Recall that $Q_\wz$ is the law
	of the Markov chain with initial state $\wz$ and transition kernel $q$.
	Then
	\begin{align*}
		&\!\!\!\!n^{-1}\log P_\wz\{L_n\in O\}\ge n^{-1}\log P_\wz\{L_n\in B\}\\
		&\!\!\!\!\ge n^{-1}\log 
			\frac{E^{Q_\wz}\Big[\Big(\frac{d{Q_\wz}_{|\sF_{n-1}}}{d{P_\wz}_{|\sF_{n-1}}}\Big)^{-1}
				\one\{L_n\in B\}\Big]}{Q_\wz\{L_n\in B\}}
			+n^{-1}\log Q_\wz\{L_n\in B\}
		\intertext{(by Jensen's inequality, applied  to $\log x$)}
		&\!\!\!\!\ge\frac{-n^{-1} E^{Q_\wz}\Big[\log\Big(\frac{d{Q_\wz}_{|\sF_{n-1}}}{d{P_\wz}_{|\sF_{n-1}}}\Big)
				\one\{L_n\in B\}\Big]}{Q_\wz\{L_n\in B\}}
			+n^{-1}\log Q_\wz\{L_n\in B\}\\
		&\!\!\!\!=\frac{-n^{-1} E^{Q_{\wz}}\Big[\log\Big(\frac{d{Q_\wz}_{|\sF_{n-1}}}{{dP_\wz}|_{\sF_{n-1}}}\Big)\Big]}{Q_\wz\{L_n\in B\}}
			+\frac{n^{-1} E^{Q_\wz}\Big[\log\Big(\frac{{dQ_\wz}_{|\sF_{n-1}}}{d{P_\wz}_{|\sF_{n-1}}}\Big)
				\one\{L_n\notin B\}\Big]}{Q_\wz\{L_n\in B\}}
			+n^{-1}\log Q_\wz\{L_n\in B\}\\
		&\!\!\!\!=
		\frac{-n^{-1}H\Big({Q_\wz}_{|\sF_{n-1}}\,\Big|\,{P_\wz}_{|\sF_{n-1}}\Big)}
		{Q_\wz\{L_n\in B\}}
			+\frac{n^{-1} E_\wz\Big[
			\frac{d{Q_\wz}_{|\sF_{n-1}}}{d{P_\wz}|_{\sF_{n-1}}}
			\log\Big(\frac{d{Q_\wz}_{|\sF_{n-1}}}{{dP_\wz}_{|\sF_{n-1}}}\Big)
				\one\{L_n\notin B\}\Big]}
			{Q_\wz\{L_n\in B\}}+n^{-1}\log Q_\wz\{L_n\in B\}\\
		&\!\!\!\!\ge
		\frac{-n^{-1}H\Big({Q_\wz}_{|\sF_{n-1}}\,\Big|\,{P_\wz}_{|\sF_{n-1}}\Big)}
		{Q_\wz\{L_n\in B\}}
		-\frac{n^{-1}e^{-1}}{Q_\wz\{L_n\in B\}}
			+n^{-1}\log Q_\wz\{L_n\in B\}.	
	\end{align*}
	In the last inequality we used $x\log x\ge-e^{-1}$.		
	Observe next that $\mu$ and $q$ satisfy the assumptions of Lemma \ref{ergodicity}.
	Thus, $Q_\wz\{L_n\in B\}$ converges to 1 for $\mu$-a.e.\ $\wz$.
	Furthermore, if we define
		\[F(\wz)=\sum_{z\in\range} q(\wz,\Sopr_z\wz)\,
			\log\frac{q(\wz,\Sopr_z\wz)}{\pr(\wz,\Sopr_z\wz)}\ge0,\text{ (by Jensen's inequality)}\]
	then $E^{\mu}[F]=H(\mu\times q\,|\,\mu\times\pr)<\infty$ and \eqref{as-erg-thm} implies that for $\mu$-a.e.\ $\wz$
		\begin{align*}
			\lim_{n\to\infty} n^{-1}H\Big({Q_\wz}_{|\sF_{n-1}}\,\Big|\,{P_\wz}_{|\sF_{n-1}}\Big)
			&=
			\lim_{n\to\infty}E^{Q_\wz}\Big[n^{-1}\sum_{k=0}^{n-1}F(\wz_k)\Big]
			=E^{\mu}[F]=H(\mu\times q\,|\,\mu\times\pr).
		\end{align*}
		
	We have thus shown that
		\[\varliminf_{n\to\infty} n^{-1}\log P_\wz\{L_n\in O\}\ge -H(\mu\times q\,|\,\mu\times\pr)\]
	for $\mu$-a.e.\ $\wz$. By Lemma \ref{ergodicity}, this is also true $P_0^{(\ell)}$-a.s. 
\end{proof}

To prove the lower bound in Theorem \ref{MC-qldp-th} we next need to remove the positivity restriction
on $q$. 
This is a simple consequence of convexity. 

\begin{proof}[Proof of the lower bound in Theorem \ref{MC-qldp-th}]
Recall our assumption \eqref{RWRE-invariant}. If an invariant measure $\P_\infty$ exists, then let
$\qhat=\pr$ and $\muhat(d\w,dz_{1,\ell})=\P_\infty(d\w)P_0^\w(dx_{0,\ell})$. If, alternatively,
$\E[|\log\pi_{0,z}|]<\infty$, for all $z\in\range$, then set
$\pihat_z=c\,e^{-|z|}/(\E[|\log\pi_{0,z}|]\vee1)$, where $c$ is chosen
so that $\sum_{z\in\range}\pihat_z=1$. This ensures that
	\begin{align*}
		\sum_z\pihat_z\,\E\Big[\log\frac{\pihat_z}{\pi_{0,z}}\Big]<\infty.
	\end{align*}
In this case, define $\muhat(d\w,dz_{1,\ell})=\P(d\w)P(dz_{1,\ell})$, where $P$ is 
an i.i.d.\ probability measure with $P\{Z_i=z\}=\pihat_z$. Let $\qhat(\wz,\Sopr_z\wz)=\pihat_z$. 

Observe that in either case, $\muhat\ll P_0^{(\ell)}$, $\muhat\qhat=\muhat$, and $H(\muhat\times\qhat\,|\,\muhat\times\pr)<\infty$. 

Let $\mu\in O$ be such that $\mu_0\ll\P$. By \eqref{RWRE-regular}, $\mu\ll P_0^{(\ell)}$.
Let $q$ be such that $\mu$ is $q$-invariant and $H(\mu\times  q\,|\,\mu\times\pr)<\infty$. 

Fix $\e\in(0,1)$ and define $\mu_\e=\e\muhat+(1-\e)\mu$. For $\e>0$ small
enough, this measure belongs to the open set $O$. It is also clear that $\mu_\e\ll P_0^{(\ell)}$. 
Let $f_\e=\frac{d\mu}{d\mu_\e}$ and $\fhat_\e=\frac{d\muhat}{d\mu_\e}$. Note that Lemma \ref{ergodicity} implies that $\muhat\sim P_0^{(\ell)}$.
Thus, $\muhat\sim\mu_\e$ and $\mu_\e\{\fhat_\e>0\}=1$.
Next, define the kernel
	\[q_\e(\wz,\Sopr_z\wz)
	=\e\fhat_\e(\wz)\qhat(\wz,\Sopr_z\wz)+(1-\e)f_\e(\wz)q(\wz,\Sopr_z\wz).\]
Then, $\mu_\e$-a.s., $\sum_{z\in\range}q_\e(\wz,\Sopr_z\wz)=1$ and $q_\e(\wz,\Sopr_z\wz)>0$ for all $z\in\range$.
Furthermore, $\mu_\e q_\e=\mu_\e$. Indeed,
	\begin{align*}
		&\sum_{z\in\range}\int G(\Sopr_z\wz)[\e \fhat_\e(\wz)\qhat(\wz,\Sopr_z\wz)+(1-\e)f_\e(\wz)q(\wz,\Sopr_z\wz)]
			\,\mu_\e(d\zeta)\\
		&=\e\sum_{z\in\range}\int G(\Sopr_z\wz) \qhat(\wz,\Sopr_z\wz)\muhat(d\wz)+
		(1-\e)\sum_{z\in\range}\int G(\Sopr_z\wz) q(\wz,\Sopr_z\wz)\mu(d\wz)\\
		&=\e\int G\,d\muhat+(1-\e)\int G\,d\mu=\int G\,d\mu_\e.
	\end{align*}
	 	
On the other hand, Jensen's inequality (applied to $x\log x$) implies
	\begin{align*}
		&H(\mu_\e\times q_\e\,|\,\mu_\e\times\pr)
		=\sum_z\int q_\e(\wz,\Sopr_z\wz)\,\log\frac{q_\e(\wz,\Sopr_z\wz)}{\pr(\wz,\Sopr_z\wz)}\,
			\mu_\e(d\wz)\\
		&\le\sum_z\int \e\fhat_\e(\wz)\, \qhat(\wz,\Sopr_z\wz)\,\log\frac{\qhat(\wz,\Sopr_z\wz)}{\pr(\wz,\Sopr_z\wz)}\,\mu_\e(d\wz)
		+\sum_z\int (1-\e) f_\e(\wz)\,q(\wz,\Sopr_z\wz)\,\log\frac{q(\wz,\Sopr_z\wz)}{\pr(\wz,\Sopr_z\wz)}\,\mu_\e(d\wz)\\ 
		&=\e\sum_z\int \qhat(\wz,\Sopr_z\wz)\,\log\frac{\qhat(\wz,\Sopr_z\wz)}{\pr(\wz,\Sopr_z\wz)}\,\muhat(d\wz)
		+(1-\e)\sum_z\int q(\wz,\Sopr_z\wz)\,\log\frac{q(\wz,\Sopr_z\wz)}{\pr(\wz,\Sopr_z\wz)}\,\mu(d\wz)\\
		&=\e H(\muhat\times \qhat\,|\,\muhat\times\pr)+(1-\e)H(\mu\times q\,|\,\mu\times\pr). 
	\end{align*}
	
	Since $H(\muhat\times\qhat\,|\,\muhat\times\pr)<\infty$, applying Lemma \ref{weak-lower-bound} and then taking $\e\to0$ proves 
	the lower bound in Theorem \ref{MC-qldp-th} with function $H_\ell$. 
	The argument above can also be used to show that $H_\ell$
	is convex.   Thus the lower bound also holds with $H_\ell^{**}$.
\end{proof}

\section{Upper bound}
\label{upper-qldp-section}
To motivate the complicated upper bound proof we first present a simple version of it that works for a finite
$\Omega$, which is the case of a periodic environment. In this case, the upper bound only requires the regularity assumption
\eqref{RWRE-regular}. Note also that the finiteness of $\Omega$ implies the existence of $\P_\infty$ as in assumption \eqref{RWRE-invariant},
and hence the lower bound (and, consequently, the large deviation principle) also holds under only \eqref{RWRE-regular}.

Fix $\ell\ge1$. Given bounded continuous functions $h$  and $f$ on $\bigom_\ell$ define
	\[K_{\ell,h}(f)=\P\text{-}\esssup_\w\sup_{z_{1,\ell}}\log \sum_z \pr(\wz,\Sopr_z\wz) e^{f(\wz)-h(\wz)+h(\Sopr_z\wz)}.\]
Define $\unK_\ell:\sC_b(\bigom_\ell)\to\R$ by
	\[\unK_\ell(f)=\inf_{h\in\sC_b(\bigom_\ell)} K_{\ell,h}(f).\]
A small modification of Donsker and Varadhan's argument in \citep{dons-vara-75-a},
given below in  Lemma \ref{lm-new-upper}, shows that for $\P$-a.e.\ $\w$
and all $z_{1,\ell}\in\range^\ell$ one has, for all compact sets $C\subset\M_1(\bigom_\ell)$, 
\begin{align*}
	\varlimsup_{n\to\infty}n^{-1}\log P_\wz\{L_n\in C\}&\le-\inf_{\mu\in C} \unK_\ell^*(\mu),
\end{align*}
where $\unK_\ell^*(\mu)=\sup_{f\in\sC_b(\bigom_-)}\{E^\mu[f]-\unK_\ell(f)\}$ is the convex conjugate of $\unK_\ell$.
Now we 
observe what it takes to turn this rate function $\unK_\ell^*$ into $H_\ell^{**}$ and thereby
match the upper and lower bounds. 

First
\beq	\begin{aligned}
	\unK_\ell(f)&=\inf_{h\in\sC_b(\bigom_\ell)}\P\text{-}\esssup_\w\sup_{z_{1,\ell}}\log \sum_z \pr(\wz,\Sopr_z\wz) e^{f(\wz)-h(\wz)+h(\Sopr_z\wz)}\\
	&=\inf_{h\in\sC_b(\bigom_\ell)}\sup_{\mu:\mu_0\ll\P} \{E^\mu[f]-E^\mu[h-\log \pr(e^h)]\}.
	\end{aligned} \label{motivK}\eeq 
	
On the other hand,  given  $\mu,\nu\in\M_1(\bigom_\ell)$,   we have this   variational formula:  
	\begin{align*}
	&\inf\{H(\alpha\,|\,\alpha_1\times\pr): \alpha\in\M_1(\bigom_\ell^2), \,
	 \alpha_1=\mu, \,\alpha_2=\nu\}
	 =\sup_{h\in\sC_b(\bigom_\ell)}\{E^\nu[h]-E^\mu[\log\pr(e^h)]\},\end{align*}
where  $\alpha_1$ and $\alpha_2$ are the first and second marginals of $\alpha$
(see Theorem 2.1 of \citep{dons-vara-76}, Lemma 2.19 of \citep{sepp-93-ptrf-a}, or Theorem 13.1 of \citep{rass-sepp-ldp}).
Out of this we get 
	\begin{align}
	H_\ell^*(f)&=\sup_{\mu:\mu_0\ll\P}\Big\{E^\mu[f]-\inf\{H(\mu\times q\,|\,\mu\times\pr):\mu q=\mu\}\Big\}\nn\\
		&=\sup_{\mu:\mu_0\ll\P}\Big\{E^\mu[f]-\inf_{\alpha\in\M_1(\bigom_\ell^2)}\{H(\alpha\,|\,\alpha_1\times\pr):\alpha_1=\alpha_2=\mu\}\Big\}\nn\\
		&=\sup_{\mu:\mu_0\ll\P}\Big\{E^\mu[f]-\sup_{h\in\sC_b(\bigom_\ell)} E^\mu[h-\log \pr(e^h)]\Big\}\nn\\
		&=\sup_{\mu:\mu_0\ll\P}\inf_{h\in\sC_b(\bigom_\ell)} \Big\{E^\mu[f]-E^\mu[h-\log \pr(e^h)]\Big\}.
\label{motivH*}	\end{align}
Comparison of \eqref{motivK} and \eqref{motivH*} shows that matching  $\unK_\ell$ and $H_\ell^*$,
and thereby completing
the  upper bound of  Theorem \ref{MC-qldp-th},  boils down to an application of a 
 minimax theorem (such as K\"onig's theorem, see \citep{kass-94} or \citep{rass-sepp-ldp}).  
  However, the set $\{\mu:\mu_0\ll\P\}$ is compact if, and only if, $\P$ has finite support. 
%
%
%
%

To get around this difficulty we abandon the attempt to prove 
  the equality of $\unK_\ell$ and $H_\ell^*$. Instead, we   redefine $\unK_\ell$ by taking infimum 
over a larger set of functions. This decreases $\unK_\ell$ and makes it possible to prove 
$H_\ell^*\ge\unK_\ell$. We will still be able to prove that $H_\ell^*\le\unK_\ell$
and that $\unK^*_\ell$ governs the large deviation upper bound. The new definition extends
the class of functions to include weak limits of $h_k(\Sopr_z\wz)-h_k(\wz)$, which
may lose this form. Such limits are the so-called ``corrector functions'', familiar
from quenched central limit theorems for random walk in random environment (see for example \citep{rass-sepp-05}
and the references therein). Let us introduce this class of functions and redefine $\unK_\ell$. 

\begin{definition}
\label{cK-def}
A measurable function $F:\bigom_\ell\times\range\to\R$ is in 
class $\K^p(\bigom_\ell\times\range)$
if it satisfies the following three conditions
\begin{itemize}
\item[(i)] Moment: for each $z_{1,\ell}\in\range^\ell$ and $z\in\range$, 
$\E[|F(\w,z_{1,\ell},z)|^p]<\infty$.
\item[(ii)] Mean zero: for all $n\ge\ell$ and $\{a_i\}_{i=1}^n\in \range^n$ the following holds.
If $\wz_0=(\w, a_{n-\ell+1,n})$ and $\wz_i=\Sopr_{a_i}\wz_{i-1}$ for $i=1,\dotsc, n$, then
\begin{align*}
\E\Big[\sum_{i=0}^{n-1} F(\wz_i,a_{i+1})\Big]=0.
\end{align*}
In other words, expectation vanishes
whenever the sequence of moves $\Sopr_{a_1},\dotsc,\Sopr_{a_n}$
 takes $(\w, z_{1,\ell})$ to $(T_x\w, z_{1,\ell})$
for all $\w$, for fixed $x$ and $z_{1,\ell}$. 
\item[(iii)] Closed loop: for $\P$-a.e.\ $\w$ and 
any two paths $\{\wz_i\}_{i=0}^n$ and 
$\{\bar\wz_j\}_{j=0}^m$ with 
$\wz_0=\bar\wz_0=(\w,z_{1,\ell})$, $\wz_n=\bar\wz_m$,  
$\wz_i=\Sopr_{a_i}\wz_{i-1}$,
and $\bar\wz_j=\Sopr_{\bar a_j}\bar\wz_{j-1}$, for $i,j>0$ and some 
$\{a_i\}_{i=1}^n\in\range^n$ and $\{\bar a_j\}_{j=1}^m\in\range^m$, 
we have
\begin{align*}
&\sum_{i=0}^{n-1} F(\wz_i,a_{i+1})
=\sum_{j=0}^{m-1} F(\bar\wz_j,\bar a_{j+1}).
\end{align*}
\end{itemize}
\end{definition}

\begin{remark}
In (iii) above, if one has a loop ($\wz_0=\wz_n$), then one can take $m=0$ and the 
right-hand side in the above display vanishes. 
\end{remark}

\begin{remark}
Note that functions $F(\wz,z)=h(\Sopr_z\wz)-h(\wz)$
belong to this class.
\end{remark}

The following sublinear growth property is crucial. We postpone its proof to the appendix.

\begin{lemma}
\label{CLASS K}
Let $(\Omega,\kS,\P,\{T_z\})$ be ergodic. 
Assume $\P$ satisfies assumptions \eqref{RWRE-compact} 
and \eqref{RWRE-span}.
Let $F\in\K^p(\bigom_\ell\times\range)$ with $p>d$ being the same
as in assumption \eqref{RWRE-moment}. 
Then, for $\P$-a.e.\ $\w$
\[\lim_{n\to\infty} n^{-1}
\sup_{z_{1,\ell}\in\rrange^{\ell}}\sup_{\substack{(a_1,\dotsc,a_n)\in\rrange^n\\
\wz_0=(\w,z_{1,\ell}),\wz_i=\Sopr_{a_i}\wz_{i-1}}}
\Big|\sum_{k=0}^{n-1} F(\wz_k,a_{k+1})\Big|=0.\]
\end{lemma}

\begin{remark}\label{explanation}
The above lemma clarifies why the method we use requires the condition $p>d$. Indeed, consider the case
$\ell=0$, $\Omega=\cP^{\Z^d}$, $\P$ a product measure, and $F(\w,z)=h(T_z\w)-h(\w)$ with $h$ being a function of just $\w_0$. 
Then, the conclusion of the lemma is that $n^{-1}\sup_{|x|\le n}|h(\w_x)|$ vanishes
at the limit. For this to happen one needs more than $d$ moments for $h$.
\end{remark}

Now, for $F\in\K^p(\bigom_\ell\times\range)$ and $f\in\sC_b(\bigom_\ell)$, redefine
	\[K_{\ell,F}(f)=\P\text{-}\esssup_\w\sup_{z_{1,\ell}}\log \sum_z \pr(\wz,\Sopr_z\wz) e^{f(\wz)+F(\wz,z)}.\]
Redefine $\unK_\ell:\sC_b(\bigom_\ell)\to\R$ by
	\begin{align*}
	\unK_\ell(f)=\inf_{F\in\K^p(\bigom_\ell\times\range)} K_{\ell,F}(f).
	\end{align*}

\begin{lemma}\label{lm-new-upper}
Assume the conclusion of Lemma \ref{CLASS K} holds.
For $\P$-a.e.\ $\w$ and all $z_{1,\ell}$,
for all compact sets $C\subset\measures(\bigom_\ell)$,
\begin{align*}
	\varlimsup_{n\to\infty}n^{-1}\log P_\wz\{L_n\in C\}&\le-\inf_{\mu\in C} \unK_\ell^*(\mu),
\end{align*}
where $\unK_\ell^*(\mu)=\sup_{f\in\sC_b(\bigom_\ell)}\{E^\mu[f]-\unK_\ell(f)\}$ is the convex conjugate of $\unK_\ell$.
\end{lemma}

\begin{proof}
Fix $\mu\in C$ and $c<\inf_C\unK^*_\ell$.   
There exist $f\in\sC_b(\bigom_\ell)$ and $F\in \K^p(\bigom_\ell\times\range)$
such that $E^\mu[f]-K_{\ell,F}(f)>c$. Fix $\e>0$ and
define the neighborhood
	\[B_\e(\mu)=\{\nu\in\M_1(\bigom_\ell):|E^\nu[f]-E^\mu[f]|<\e\}.\]


Lemma \ref{CLASS K}
implies that for $\P$-a.e.\ $\w$ there exists a finite $c_\e(\w)>0$ such that for all $n$ and $z_{1,\ell}\in\range^\ell$, 
\[\sum_{k=0}^{n-1}F(\wz_k,Z_{k+\ell+1})\ge-c_\e-n\e,\quad P_{\wz}\text{-a.s.}\]
Therefore, for all $z_{1,\ell}\in\range^\ell$ and $\P$-a.e.\ $\w$,
\begin{align*}
	&P_\wz\{L_n\in B_\e\}=E_\wz[e^{nL_n(f)}e^{-nL_n(f)}\one\{L_n\in B_\e\}]\\
		&\le e^{c_\e+n\e}\,e^{-n\inf_{\nu\in B_\e}E^\nu[f]}\,E_{\wz}\Big[\exp\Big\{-c_\e-n\e+{\sum_{k=0}^{n-1}f(\wz_k)}\Big\}\Big]\\
		&\le
		e^{c_\e+n\e}\,e^{-nE^\mu[f]}e^{n\e}\,E_{\wz}\Big[\exp\Big\{{\sum_{k=0}^{n-1}\Big(f(\wz_k)+F(\wz_k,Z_{k+\ell+1})\Big)}\Big\}\Big]\\
		&=e^{c_\e+n\e}\,e^{-nE^\mu[f]}e^{n\e}\,
		 E_{\wz}\bigg[E_{\wz}\Big[\exp\Big\{{\sum_{k=0}^{n-1}\Big(f(\wz_k)+F(\wz_k,Z_{k+\ell+1})\Big)}\Big\}\,\Big|\,\wz_i:i\le{n-1}\Big]\bigg]\\
		&=e^{c_\e+n\e}\,e^{-nE^\mu[f]}e^{n\e}\,
		 E_{\wz}\Big[\exp\Big\{{\sum_{k=0}^{n-2}\Big(f(\wz_k)+F(\wz_k,Z_{k+\ell+1})\Big)}\Big\}
		E_{\wz_{n-1}}[e^{f(\wz_0)+F(\wz_0,Z_{\ell+1})}]\Big]\\
		&\le e^{c_\e+n\e}\,e^{-nE^\mu[f]}e^{n\e}\,e^{K_{\ell,F}(f)}\,E_{\wz}\Big[\exp\Big\{{\sum_{k=0}^{n-2}\Big(f(\wz_k)+F(\wz_k,Z_{k+\ell+1})\Big)}\Big\}\Big]\\
		&\le\cdots\le e^{c_\e+n\e}\,e^{-nE^\mu[f]}e^{n\e}\,e^{n K_{\ell,F}(f)}\le e^{c_\e+2n\e-cn}.
\end{align*}

	Since $C$ is compact, it can be covered by a finite collection of $B_\e({\mu_i})$'s and
	 	\[\varlimsup_{n\to\infty} 
			n^{-1}\log P_\wz\{L_n\in C\}\le -c+2\e.\]
	Thus, taking $\e\to0$ and $c$ to $\inf_C\unK_\ell^*$ proves the lemma for a compact $C$.
%
\end{proof}

Our next theorem gives the connection  between $\unK_\ell$ and $H_\ell$.

\begin{theorem}\label{H*=unK}
Same assumptions as in Theorem \ref{MC-qldp-th}. Then, $H_\ell^{*}\equiv\unK_\ell$ for all $\ell\ge1$.
\end{theorem}

We are now ready to prove the above theorem and finish the proof of Theorem \ref{MC-qldp-th}.

\begin{proof}[Proof of Theorem \ref{H*=unK} and the upper bound in Theorem \ref{MC-qldp-th}]
It suffices to prove that for bounded continuous functions $f$, 
\begin{align}\label{tempo-identity}
\unK_\ell(f)\le H_\ell^*(f)=\sup_\mu\{E^\mu[f]-H_\ell(\mu)\}.
\end{align}
Indeed, this would imply that $H^{**}_\ell\le\unK^*_\ell$ and Lemma \ref{lm-new-upper} implies then the upper bound
in Theorem \ref{MC-qldp-th}. Furthermore, due to the lower bound and the uniqueness of the rate function (see Theorem 2.18  of \citep{rass-sepp-ldp}), 
we in fact have that $H^{**}_\ell=\unK_\ell^*$.
This implies that $H^*_\ell=\unK_\ell^{**}$ and since $\unK_\ell$ is convex and continuous in the uniform norm, we have that $H^*_\ell=\unK_\ell$.

Let us now prove \eqref{tempo-identity}. This is trivial when $H_\ell^*(f)=\infty$. Assume thus that $H_\ell^*(f)<\infty$.

Let $\kS_k$ be an increasing sequence of  finite $\sigma$-algebras on $\Omega$, generating $\kS$.
Assume that for all $k\ge1$ and $y\in\range$, $T_y\kS_{k-1}\subset\kS_k$. 
Let $\M_1^k=\M_1^k(\bigom_\ell)$ be the set of probability measures $\mu$ on $\bigom_\ell$ such that $\mu_0\ll\P$ and
$\frac{d\mu_0}{d\P}$ is $\kS_k$-measurable. 
Now write
\begin{align*}
	H_\ell^*(f)&=\sup_{\mu:\mu_0\ll\P}\{E^\mu[f]-H_\ell(\mu)\}
	\ge\sup_{\mu\in\M_1^k}\{E^\mu[f]-H_\ell(\mu)\}.
\end{align*}

To conclude the proof of \eqref{tempo-identity}, one invokes the following lemma.

\begin{lemma}\label{A>K}
Same assumptions as in Theorem \ref{MC-qldp-th}. Fix $\ell\ge1$ and let $f\in\sC_b(\bigom_\ell)$ and 
$A<\infty$ be such that 
	\[A\ge\sup_{\mu\in\M_1^k}\{E^\mu[f]-H_\ell(\mu)\},\]
for all $k\ge1$. Then, $A\ge\unK_\ell(f)$.
\end{lemma}

\begin{proof}[Proof of Lemma \ref{A>K}]
Let $\M_1^{k,2}$ be the set of probability measures $\alpha$ on $\bigom_\ell^2$ such that 
the first $\bigom_\ell$-marginal $\alpha_1\in\M_1^k$.  
Observe next that if $\alpha\in\M_1(\bigom_\ell^2)$ is such that $\alpha_1\ne\alpha_2$, then
 	\[\inf_{h\in\sC_b(\bigom_\ell)}\{E^{\alpha_2}[h]-E^{\alpha_1}[h]\}=-\infty.\]
Write
\begin{align*}
	A&\ge\sup_{\mu\in\M_1^k}\Big\{E^\mu[f]-\inf\{H(\mu\times q\,|\,\mu\times\pr):\mu q=\mu\}\Big\}\\
		&=\sup\Big\{E^{\alpha_1}[f]-H(\alpha\,|\,\alpha_1\times\pr):\alpha\in\M_1^{k,2},\alpha_1=\alpha_2\Big\}\nn\\
		&\ge \sup_{\alpha\in\M_1^{k,2}}\inf_{h\in\sC_b(\bigom_\ell)} \Big\{  E^{\alpha_1}[f]+E^{\alpha_2}[h]-E^{\alpha_1}[h]
		-H(\alpha\,\vert\,\alpha_1\times \pr)\Big\}.\nn
\end{align*}

Since the quantity in braces is linear (and hence continuous and convex) in $h$ and concave and upper semicontinuous in $\alpha$,
and since $\M_1^{k,2}$ is compact, we can apply K\"onig's minimax theorem; see \citep{kass-94}. 
%
Then 
\begin{align*}
	A&\ge\inf_{h\in\sC_b(\bigom_\ell)}\sup_{\alpha\in\cM_1^{k,2}} \Big\{ E^{\alpha_1}[f]+E^{\alpha_2}[h]-E^{\alpha_1}[h]
		-H(\alpha\,\vert\,\alpha_1\times \pr)\Big\} \\
		&=\inf_{h\in\sC_b(\bigom_\ell)}\sup_{\mu\in\cM_1^k} \sup_{q\in\MC(\bigom_\ell)} \int  \Big[ f(\wz)+ qh(\wz)-h(\wz)
			-H(q(\wz,\cdot)\,|\, \pr(\wz,\cdot))\Big]\,\mu(d\wz)  \\
		&=\inf_{h\in\sC_b(\bigom_\ell)}\sup_{\mu\in\cM_1^k}  E^\mu[ f -h + \log \pr(e^h) ]\\
		&\ge\inf_{h\in\sC_b(\bigom_\ell)}\sup_{z_{1,\ell}\in\range^\ell}\sup\Big\{E^\mu[ f -h + \log \pr(e^h) ]:
		\mu=\mu_0\otimes\delta_{z_{1,\ell}}\text{ and }\frac{d\mu_0}{d\P}\text{ is $\kS_k$ measurable}\Big\}.
\end{align*}
In the last equality above we passed the sup under the integral, since the integrand is a function of $q(\wz,\cdot)$ and one
can maximize for each $\wz$ separately.  Then we used the variational characterization of relative entropy; see 
Lemma 10.1 in \citep{vara-ldp} or Theorem 6.7 in \citep{rass-sepp-ldp}.
We thus have
\begin{align*}
		A&\ge\inf_{h\in\sC_b(\bigom_\ell)}\sup_{z_{1,\ell}\in\range^\ell}  \P\text{-}\esssup_\w \E[ f -h + \log \pr(e^h) \,|\, \kS_k ]\\
			&=\inf_{h\in\sC_b(\bigom_\ell)}\sup_{z_{1,\ell}\in\range^\ell}  \P\text{-}\esssup_\w \E\Big[  \log\sum_z \pr(\wz, \Sopr_z\wz) 
e^{f(\wz)-h(\wz)+h(\Sopr_z\wz) }  \,\Big|\, \kS_k \Big]. 
\end{align*}

Let $\nu\in\M_1(\range)$ with $\nu(z)>0$ for all $z\in\range$. Write the last conditional expectation as 
\[   \E\Bigl[  \log\sum_z \nu(z) 
\exp\Bigl\{\log [\nu(z)^{-1} \pr(\wz, \Sopr_z\wz)]  + f(\wz)-h(\wz)+h(\Sopr_z\wz) \Bigr\}  \,\Big|\, \kS_k \Bigr].  \]
An application of an infinite-dimensional version of Jensen's inequality (see Lemma \ref{infinite-Jensen}) and cancelling the $\nu(z)$-factors gives
\[A\ge \inf_{h\in\sC_b(\bigom_\ell)} \sup_{z_{1,\ell}\in\range^\ell} \P\text{-}\esssup_\w \Bigl\{ \log \sum_z 
e^{\E[ \log \pr(\wz, \Sopr_z\wz)  + f(\wz)-h(\wz)+h(\Sopr_z\wz)\,|\, \kS_k]} \Bigr\}.\]

The above means that for $\e>0$ and $k\ge1$ there exists a bounded continuous
function $h_{k,\e}$ such that for all $z_{1,\ell}\in\range^\ell$ and $\P$-a.s.
\begin{align}
\label{F-bound}
A+\e\ge
 \log\sum_z e^{\E[f(\wz)+\log \pr(\wz,\Sopr_z\wz)-h_{k,\e}(\wz)+h_{k,\e}(\Sopr_z\wz)\,|\,\kS_k]}\,.
 \end{align}

Next, we show that the sequence 
\begin{align}
\label{F}
F_{k,\e}(\wz,z)=\E[h_{k,\e}(\Sopr_z\wz)-h_{k,\e}(\wz)\,|\,\kS_{k-1}]
\end{align}
is uniformly bounded in $L^p(\P)$, for any fixed $z_{1,\ell}$ and $z$. 
Hence, along a subsequence, $F_{k,\e}$ converges in the $L^p(\P)$ weak topology 
to some $F_\e\in L^p(\P)$. We can in fact use the same
subsequence for all $z_{1,\ell}$ and $z$.
We will still call this subsequence $(F_{k,\e})$.
One can also directly check that $F_\e\in\cK^p(\bigom_\ell\times\range)$. 
In order not to interrupt the flow we postpone the proof of these two facts to 
Lemma \ref{F-lemma} below. 

On the other hand, 
\[M_k(\wz,z)=\E[f(\wz)+\log \pr(\wz,\Sopr_z\wz)\,|\,\kS_{k-1}]\] 
is a martingale whose $L^{p}(\P)$-norm is uniformly bounded. It thus
converges in $L^p(\P)$ (as well as almost-surely) 
to $f(\wz)+\log \pr(\wz,\Sopr_z\wz)$, for all $z_{1,\ell}$ and $z$. 
Thus, by Theorem 3.13 of \citep{rudi-func-anal-91}, for each fixed $z_{1,\ell}$ and $z$,
there exists a sequence of random variables $g_{k,\e}(\wz,z)$ that converges 
strongly in $L^p$ (and thus a subsequence converges
$\P$-a.s.) to 
$f(\wz)+\log \pr(\wz,\Sopr_z\wz)+F_\e(\wz,z)$ and such that 
$g_{k,\e}$ is a convex combination of
$\{M_j+F_{j,\e}:j\le k\}$. 
One can then extract a further subsequence that converges
$\P$-a.s.\ for all $z_{1,\ell}$ and $z$. 

By Jensen's inequality, 
we have for all $z_{1,\ell}\in\range^\ell$ and $\P$-a.s.
\begin{align*}
e^{A+\e}
&\ge\sum_{z\in\range}\E\Big[e^{\E[f(\wz)+\log\pr(\wz,\Sopr_z\wz)-h_{k,\e}(\wz)+h_{k,\e}(\Sopr_z\wz)|\kS_k]}\,\Big|\,\kS_{k-1}\Big]
\ge \sum_{z\in\range}e^{M_k(\wz,z)+F_{k,\e}(\wz,z)}.
\end{align*}
Since this is valid for all $k\ge1$, another application of Jensen's inequality gives
\begin{align*}
e^{A+\e}\ge \sum_{z\in\range}e^{g_{k,\e}(\wz,z)}.
\end{align*}
Taking $k\to\infty$ implies, for $\P$-a.e.\ $\w$ and all $z_{1,\ell}\in\range^\ell$,
\[A+\e\ge f(\wz)+\log\sum_{z\in\range}\pr(\wz,\Sopr_z\wz)e^{F_\e(\wz,z)}\]
and thus
\[A+\e\ge \inf_{F\in\K^p(\bigom_\ell\times\range)}
\sup_{z_{1,\ell}\in\range^\ell}\P\text{-}\esssup_\wz\Big\{f(\wz)+\log\sum_{z\in\range}\pr(\wz,\Sopr_z\wz)e^{F(\wz,z)}\Big\}.\]
Taking $\e\to0$ proves that $A\ge\unK_\ell$. 
\end{proof}

\begin{lemma}\label{F-lemma}
Assume $(\Omega,\kS,\P,\{T_z\})$ is ergodic. Assume $\P$ satisfies assumptions \eqref{RWRE-compact}, 
\eqref{RWRE-elliptic}, and \eqref{RWRE-moment}.
Then, for $\e>0$, $z_{1,\ell}\in\range^\ell$, and $z\in\range$, 
\begin{align}\label{unif-Lp}
\sup_k\E[|F_{k,\e}(\w,z_{1,\ell},z)|^p]<\infty.
\end{align}
Moreover, if a subsequence converges {\rm(}in weak $L^p(\P)$-topology{\rm)},
for each $z_{1,\ell}$ and $z$, to a limit $F_\e$, then $F_\e$ belongs to class $\K^p(\bigom_\ell\times\range)$.
\end{lemma}

\begin{proof}
Since $f$ is bounded \eqref{F-bound} implies that for $\P$-a.e.\ $\w$ and for all $z_{1,\ell}\in\range^\ell$
	\[F_{k,\e}(\wz,z)\le C-\E[\log\pr(\wz,\Sopr_z\wz)\,|\,\kS_{k-1}].\]
The $L^p(\P)$-norm of the right-hand-side is bounded by $C+\E[|\log\pi_{0,z}|^p]$, which is finite by assumption \eqref{RWRE-moment}.

By assumption \eqref{RWRE-elliptic}, there exist $a_1,\dotsc,a_m\in\range$  such that $x_\ell+z+a_1+\cdots+a_{m-\ell}=0$ and $a_{m-\ell+1,m}=z_{1,\ell}$. 
Then, letting $\wz_0=\Sopr_z\wz$ and
$\wz_{i+1}=\Sopr_{a_{i+1}}\wz_i$ and defining $(y_i,z^i_{1,\ell})$ such that $\wz_i=(T_{y_{i}}\w,z^{i}_{1,\ell})$,
 $0\le i\le m-1$, we have
 	\begin{align*}
	&\sum_{i=0}^{m-1}\E[h_{k,\e}(\Sopr_{a_{i+1}}\wz_i)-h_{k,\e}(\wz_i)\,|\,T_{-y_i}\kS_k]\\
	&\qquad=\sum_{i=0}^{m-1}\E[h_{k,\e}(T_{y_i+z_1^i}\w,\Sopr_{a_{i+1}}z^i_{1,\ell})-h_{k,\e}(T_{y_i}\w,z^i_{1,\ell})\,|\,T_{-y_i}\kS_k]\\
	&\qquad=\sum_{i=0}^{m-1} \E[h_{k,\e}(T_{z_1^i}\w,\Sopr_{a_{i+1}}z^i_{1,\ell})-h_{k,\e}(\w,z^i_{1,\ell})\,|\,\kS_k]\circ T_{y_i}\\
	&\qquad=\sum_{i=0}^{m-1} \E[h_{k,\e}(\Sopr_{a_{i+1}}(\w,z^i_{1,\ell}))-h_{k,\e}(\w,z^i_{1,\ell})\,|\,\kS_k]\circ T_{y_i}\\
	&\qquad\le C\,m-\sum_{i=0}^{m-1} \E[\log\pr((\w,z^i_{1,\ell}),\Sopr_{a_{i+1}}(\w,z^i_{1,\ell}))\,|\,\kS_k]\circ T_{y_i}.
	\end{align*}
The last inequality is a result of \eqref{F-bound}. Taking conditional expectations given $\kS_{k-1}$ one has
	\begin{align*}
	-F_{k,\e}(\wz,z)&=\E[h_{k,\e}(\wz)-h_{k,\e}(\Sopr_z\wz)\,|\,\kS_{k-1}]\\
	&=\sum_{i=0}^{m-1}\E[h_{k,\e}(\Sopr_{a_{i+1}}\wz_i)-h_{k,\e}(\wz_i)\,|\,\kS_{k-1}]\\
	&=\sum_{i=0}^{m-1}\E\Big[\E[h_{k,\e}(\Sopr_{a_{i+1}}\wz_i)-h_{k,\e}(\wz_i)\,|\,T_{-y_i}\kS_k]\,\Big|\,\kS_{k-1}\Big]\\
	&\le C\,m-\sum_{i=0}^{m-1} \E\Big[\E[\log\pr((\w,z^i_{1,\ell}),\Sopr_{a_{i+1}}(\w,z^i_{1,\ell}))\,|\,\kS_k]\circ T_{y_i}\,\Big|\,\kS_{k-1}\Big].
	\end{align*}
The $L^p(\P)$-norm of the right-hand-side is  bounded by $(C+\E[|\log\pi_{0,z}|^p])m$, which is finite by assumption \eqref{RWRE-moment}.

Consider next a weakly convergent subsequence. We will still
denote it by $F_{k,\e}$. Let $F_\e$ be its limit. 
Clearly, $F_\e\in L^p(\P)$ and the moment condition (i) in Definition \ref{cK-def} is satisfied. 
Also, since the mean zero property  (ii), in Definition \ref{cK-def}, is satisfied for each $F_{k,\e}$, it is satisfied for $F_\e$.

Furthermore, weak convergence in $L^p(\P)$ and finiteness of the $\si$-algebras $\kS_j$ imply that for any fixed $j$,
$\E[F_{k,\e}\,|\,\kS_j]$ converges to $\E[F_\e\,|\,\kS_j]$ for every $z_{1,\ell}\in\range^\ell$ and
$\P$-a.e.\ $\w$. Since the closed loop property holds for every $F_{k,\e}$, we have that for any two paths $\{\wz_i\}_{i=0}^n$ and $\{\bar\wz_j\}_{j=0}^m$ 
as in (iii) of Definition \ref{cK-def},
\begin{align*}
&\E\Big[\sum_{i=0}^{n-1} F_\e(\wz_i,a_{i+1})\,\Big|\,\kS_j\Big]
=\E\Big[\sum_{j=0}^{m-1} F_\e(\bar\wz_j,\bar a_{j+1})\,\Big|\,\kS_j\Big].
\end{align*}
Taking $j\to\infty$ and using the martingale convergence theorem proves the closed loop property holds for $F_\e$.
\end{proof}

The proof of Theorems \ref{H*=unK} and  \ref{MC-qldp-th} is thus complete.
\end{proof}

\section{Proof of Theorem \ref{main}}\label{proof-section}
We will now present the proof of the main theorem.
Note first that for all $k\ge0$ and $\P$-a.e.\ $\w$, \[P_0^\w\{\Sopr(T_{X_k}\w,Z_{k+1,\infty})=(T_{X_{k+1}}\w,Z_{k+2,\infty})\}=1.\]
Thus, the empirical measure $R_n^{1,\infty}$ comes deterministically close to the set of $\Sopr$-invariant measures
and every non-$\Sopr$-invariant measure has a neighborhood that has zero probability for all large enough $n$.
Since the set of such measures is open and function $H$ in Theorem \ref{main} is infinite on it, we need not be concerned with them.

Recall definitions \eqref{Helldef} of $H_\ell$ 
and \eqref{Hdef} of $H$. 
%
Now, Lemma \ref{transfer of qldp} and Theorem \ref{MC-qldp-th} imply that an almost-sure level 3 large deviation principle 
holds with rate function $\sup_{\ell\ge1}H_\ell^{**}$.
It remains to identify this rate function with the one in the statement of Theorem \ref{main}. 
This is shown in the next lemma.

	\begin{lemma}\label{conjugacy}
	Assume $\P$ is invariant for the shifts $\{T_z\}$ and satisfies assumptions \eqref{RWRE-regular} and \eqref{RWRE-span}.
	If $\mu\in\M_1(\bigom_+)$ is $\Sopr$-invariant, then
		\begin{align}\label{claim1}
		\sup_{\ell\ge1}H_\ell(\mu_{|\bigom_\ell})=H(\mu).
		\end{align}
	In particular, $H$ is convex. If, furthermore, 
	the compactness assumption \eqref{RWRE-compact} holds then 
		\begin{align}\label{claim2}
		\sup_{\ell\ge1}H_\ell^{**}(\mu_{|\bigom_\ell})=H^{**}(\mu).
		\end{align}
	\end{lemma}
	
	\begin{proof}
		Let us start with the first identity.
		Assume $\mu_0\ll\P$ since otherwise the equality holds trivially.
		Let $\mu_-^{(\ell)}$ be the law of $(\w,Z_{1-\ell,0})$ under $\mu_-$. Then,
		the $\Sopr$-invariance of $\mu$ implies that 
		\begin{align*}
		H_\infty(\mu)&\overset{\text{def}}=\sup_{\ell\ge1}H_\ell(\mu_{|\bigom_\ell})
		=\sup_{\ell\ge1}\inf\{H(\mu_-^{(\ell)}\times q\,|\,\mu_-^{(\ell)}\times\pl):q\in\MC(\bigom_\ell)\text{ and }\mu_-^{(\ell)}q=\mu_-^{(\ell)}\}.
		\end{align*}
	
		Recall the universal kernel $\quniv$ that corresponds to all $\Sopr$-invariant measures $\mu\in\M_1(\bigom_+)$.
		Let \[\quniv^{(\ell)}(\zeta,\Sopl_z\zeta)=E^{\mu_-}[\quniv(\wz,\Sopl_z\wz)\,|\,\wz_{1-\ell,0}=\zeta].\]
		Then, $\mu_-^{(\ell)}$ is $\quniv^{(\ell)}$-invariant. Moreover, $\mu_-^{(\ell)}\times \quniv^{(\ell)}$ is the restriction of $\mu_-\times\quniv$
		to $\bigom_\ell^2$. Thus, $H(\mu_-^{(\ell)}\times \quniv^{(\ell)}\,|\,\mu_-^{(\ell)}\times\pl)\le H(\mu_-\times \quniv\,|\,\mu_-\times\pl)$.
		This shows that $H_\infty(\mu)\le H(\mu)$.
		
		The other direction is trivial if $H_\infty(\mu)=\infty$. On the other hand, if $H_\infty(\mu)=h<\infty$, then there
		exists a sequence $q_-^{(\ell)}\in\MC(\bigom_\ell)$ such that $\mu_-^{(\ell)}$ is $q_-^{(\ell)}$-invariant, and
			\[H(\mu_-^{(\ell)}\times q_-^{(\ell)}\,|\,\mu_-^{(\ell)}\times \pl)\le h+\ell^{-1}.\] 
		This implies that, for $\mu_-^{(\ell)}$-a.e.\ $\wz\in\bigom_\ell$, $q_-^{(\ell)}(\wz,\{\Sopl_z\wz:z\in\range\})=1$. 
		
		For $\ell\ge\ell'$, measures $\mu_-^{(\ell)}\times q_-^{(\ell)}$ have marginals $\mu_-^{(\ell')}$.
		Thus, for $\ell'$ fixed, measures $\mu_-^{(\ell)}\times q_-^{(\ell)}$ restricted to $\bigom_{\ell'}^2$ are tight. 
		We can use the diagonal trick to extract one sequence that converges weakly on all spaces $\bigom_{\ell'}^2$ simultaneously. 
		By Kolmogorov's extension theorem one can find a limit point $Q\in\M_1(\bigom_-^2)$. The marginals of $Q$ are equal to $\mu_-$
		and hence the conditional distribution of the second coordinate $\zeta$ under $Q$, given the first coordinate $\wz$,  defines
		a kernel $q_-(\wz,d\zeta)$ that leaves $\mu_-$ invariant.
		The following entropy argument shows that $q_-$ is still supported on $\Sopl_z$-shifts; that is
			\begin{align}\label{qbar}
			Q\Big\{(\wz,\zeta)\in\bigom_-:\zeta\in \cup_z \{\Sopl_z\wz\}\Big\} = 1.
			\end{align}		
		
		For any $\e>0$, there exists
		a compact subset $K_\e\subset\Omega$ such that $\mu_-(K_\e\times\range^{\Z_-})\ge1-\e$.
		On the other hand, for any finite $\cA\subset\range$  the function $\w\mapsto F(\w,\cA)=\sum_{z\in\cA}\pi_{0,z}(\w)$  is continuous.
		Furthermore, this function increases up to 1, for all $\w$, as $\cA$ increases to $\range$.
		Thus, for each $\w$ in $K_\e$ choose a set $\cA$ so that $F(\w,\cA)\ge 1-\e/2$ and pick an open neighborhood
		$G$ of $\w$ 
		so that for $\w'\in G$, $F(\w' ,\cA)\ge1-\e$. Since $K_\e$ is compact, it can be covered with finitely many such neighborhoods.
		Let $\cA_\e$ be the union of the corresponding sets $\cA$.
		Then, $\cA_\e$ is finite and $F(\w,\cA_\e)\ge1-\e$ for all $\w\in K_\e$. In fact, we can and will choose $\cA_\e$ to increase to $\range$
		as $\e$ decreases to 0.
		
		Now recall the variational characterization of relative entropy (see 
		Lemma 10.1 in \citep{vara-ldp} or Theorem 6.7 in \citep{rass-sepp-ldp}) and write
			\begin{align*}
				h+1&\ge H(\mu_-^{(\ell)}\times q^{(\ell)}_-\,|\,\mu_-^{(\ell)}\times \pl)\\
				&=E^{\mu_-^{(\ell)}}\Big[ \sup_f  \Big\{ \sum_z q_-^{(\ell)}(\wz, \Sopl_z\wz)f(\Sopl_z\wz)
				- \log \sum_z \pi_{0,z}(\w) e^{f(\Sopl_z\wz)} \Big\}\Big]\\
				&\ge E^{\mu_-^{(\ell)}}\Big[ \sup_f  \Big\{ \sum_z q_-^{(\ell)}(\wz, \Sopl_z\wz)f(\Sopl_z\wz)
				- \log \sum_z \pi_{0,z}(\w) e^{f(\Sopl_z\wz)} \Big\}\one_{K_\e\times\range^\ell}(\wz)\Big]\\
				&\ge E^{\mu_-^{(\ell)}}\Big[\Big\{C\sum_{z\notin\cA_\e} q_-^{(\ell)}(\wz, \Sopl_z\wz)
				- \log \sum_z \pi_{0,z}(\w) e^{C\one\{z\notin\cA_\e\}} \Big\}\one_{K_\e\times\range^\ell}(\wz)\Big]\\
				&= E^{\mu_-^{(\ell)}}\Big[\Big\{C\sum_{z\notin\cA_\e} q_-^{(\ell)}(\wz, \Sopl_z\wz)
				- \log \Big(1+(e^C-1)\sum_{z\notin\cA_\e} \pi_{0,z}(\w)\Big)\Big\}\one_{K_\e\times\range^\ell}(\wz)\Big]\\
				&\ge E^{\mu_-^{(\ell)}}\Big[\Big\{C\sum_{z\notin\cA_\e} q_-^{(\ell)}(\wz, \Sopl_z\wz)
				- \log(1+(e^C-1)\e)\Big\}\one_{K_\e\times\range^\ell}(\wz)\Big].
		\end{align*}
		In the third inequality we used $f(\wz)=C\one\{z_0\notin\cA_\e\}$. Now, fix a $\delta>0$ and choose
		$C$ large such that $(1+h)/C<\delta/2$. Then choose $\e>0$ small such that 
		$C^{-1}\log(1+(e^C-1)\e)+\e<\delta/2$. The above inequalities then become
			\begin{align*}
				&E^{\mu_-^{(\ell)}}\Big[\sum_{z\notin\cA_\e} q_-^{(\ell)}(\wz_{1-\ell,0}, \Sopl_z\wz_{1-\ell,0})\Big]
				\le C^{-1}(1+h)
				+C^{-1}\log(1+(e^C-1)\e)+1-\mu_-^{(\ell)}(K_\e\times\range^\ell)<\delta.\nn
			\end{align*}
		

		Since $\{(\wz,\zeta)\in\bigom_-^2:\zeta_{1-\ell',0}=\Sopl_z\wz_{1-\ell',0}\text{ and }z\in\cA_\e\}$ is closed it follows that 
			\begin{align*}
				&Q\{(\wz,\zeta)\in\bigom_-^2:\zeta_{1-\ell',0}=\Sopl_z\wz_{1-\ell',0}\text{ and }z\in\cA_\e\}\\
				&\ge\varlimsup_{\ell\to\infty}\mu_-^{(\ell)}\times q_-^{(\ell)}\{(\wz_{1-\ell,0},\zeta_{1-\ell,0})\in\bigom_\ell^2:\zeta_{1-\ell',0}=\Sopl_z\wz_{1-\ell',0}\text{ and }z\in\cA_\e\}\\
				&\ge\varlimsup_{\ell\to\infty}\mu_-^{(\ell)}\times q_-^{(\ell)}\{(\wz_{1-\ell,0},\Sopl_z\wz_{1-\ell,0}):\wz_{1-\ell,0}\in\bigom_{\ell},z\in\cA_\e\}
				\ge1-\delta.
			\end{align*}
		But $\{(\wz,\zeta):\wz\in\bigom_-,\ \zeta=\Sopl_z\wz\text{, and }z\in\cA_\e\}$ is equal to the decreasing limit
			\[\bigcap_{\ell'\ge1}\{(\wz,\zeta)\in\bigom_-^2:\zeta_{1-\ell',0}=\Sopl_z\wz_{1-\ell',0}\text{ and }z\in\cA_\e\}.\]
		Now, taking $\delta\to0$ then $\e\to0$ proves \eqref{qbar}. Since there 
		is a unique kernel that leaves 
		$\mu_-$ invariant 
		and is supported on $\Sopl_z$-shifts, $Q=\mu_-\times\quniv$ is the only possible 
		limit point. Lower semicontinuity of the entropy implies that
		 	\begin{align*}
			H\Big((\mu_-\times\quniv)_{|\bigom_{\ell'}^2}\,\Big|\,(\mu_-\times \pl)_{|\bigom_{\ell'}^2}\Big)
			&\le\varliminf_{\ell\to\infty}H\Big((\mu_-^{(\ell)}\times q_-^{(\ell)})_{|\bigom_{\ell'}^2}\,\Big|\,(\mu_-^{(\ell)}\times \pl)_{|\bigom_{\ell'}^2}\Big)\\
			&\le\varliminf_{\ell\to\infty}H(\mu_-^{(\ell)}\times q_-^{(\ell)}\,|\,\mu_-^{(\ell)}\times \pl)
			\le h.
			\end{align*}
		Taking $\ell'\to\infty$ proves that $H(\mu)\le H_\infty(\mu)$ and \eqref{claim1} holds.

		Next, we prove \eqref{claim2}. 
		First, we show that for $\mu\in \M_1(\bigom_\ell)$,
		\begin{align}
		H_\ell(\mu)= \inf\{H(\nu): \nu\text{ is $\Sopr$-invariant and }\nu_{\vert{\bigom_\ell}} = \mu  \}. \label{step1}
		\end{align}

		If $\mu_0\not\ll \P$  then both sides are infinite.  Suppose $\mu_0\ll \P$.
		Write temporarily $I(\nu)=\sup_\ell H^{**}_\ell(\nu_{\vert{\bigom_\ell}})$
		for the level 3 rate function.  If  $\nu$ is $\Sopr$-invariant and  $\nu_0\ll \P$
		then by \eqref{claim1}
		\begin{align}
		I(\nu)=\sup_\ell H^{**}_\ell(\nu_{\vert{\bigom_\ell}})=   \sup_\ell H_\ell(\nu_{\vert{\bigom_\ell}})  =   H(\nu).\label{H1}
		\end{align}

		By the  level 3 to  level 2 contraction,
		\[H_\ell(\mu) =  H^{**}_\ell(\mu) =   \inf\{I(\nu): \nu\text{ is $\Sopr$-invariant and }\nu_{\vert{\bigom_\ell}} = \mu  \}.\]
		Since $I\le H$,  to prove \eqref{step1}  it suffices to consider
		the case $H_\ell(\mu) < \infty$.   Only $\Sopr$-invariant measures have
		finite level 3 rate, hence there exists at least one
		$\Sopr$-invariant $\nu$ such that  $\nu_{\vert{\bigom_\ell}} = \mu$.
		Furthermore,  the measures $\nu$ that appear in the contraction satisfy
		$\nu_0=\mu_0\ll \P$, and so by \eqref{H1} equation \eqref{step1} follows.

		Now, consider  $\Sopr$-invariant measures $\nu$.
		By \eqref{claim1} $I\le H$, and since $I$ is a l.s.c.\ convex function, also  $I\le H^{**}$.
		By \eqref{step1} and the basic Lemma \eqref{lm:lsc},
		\[H^{**}_\ell(\mu)=\inf\{H^{**}(\nu):\nu\text{ is $\Sopr$-invariant and }\nu_{\vert{\bigom_\ell}} = \mu\}.\]
		Outside $\Sopr$-invariant measures  $H^{**}\equiv\infty$  so whether or not
		the invariance condition is included in the infimum is immaterial.

		Let  $c > I(\nu)$.  For each $\ell$ use above to find $\mu^{(\ell)}$ such that
		$\mu^{(\ell)}_{\vert{\bigom_\ell}} = \nu_{\vert{\bigom_\ell}}$  and
		$H^{**}(\mu^{(\ell)}) < c$.    $\mu^{(\ell)} \to \nu$  and so by lower
		semicontinuity   $H^{**}(\nu) \le  \varliminf  H^{**}(\mu^{(\ell)}) \le c$.
		This shows $H^{**} \le I$.
	\end{proof}

\appendix
\section{Technical Lemmas}
\begin{lemma}\label{infinite-Jensen}  
Let $g$ be a bounded measurable function on a product space
$\X\times\Y$, $\mu$ a probability measure on $\X$ and $\rho$ a probability measure on $\Y$.
Then
\[   \log \int_\X e^{\int_\Y g(x,y)\,\rho(dy)}\,\mu(dx) 
\le  \int_\Y\Big[ \, \log \int_\X e^{  g(x,y)}\,\mu(dx) \,\Big]\, \rho(dy). \]
\end{lemma} 
\begin{proof}  The inequality can be thought of as an infinite-dimensional Jensen's inequality, 
applied to
the convex functional  $\Psi(f)= \log \int_\X e^{  f(x)}\,\mu(dx)$.  Proof is immediate from
the variational characterization of relative entropy; see Lemma 10.1 in \citep{vara-ldp} or Theorem 6.7 in \citep{rass-sepp-ldp}.
First for an arbitrary  probability measure 
$\gamma$ on $\X$, 
\begin{align*}
& \int_\Y\Big[ \, \log \int_\X e^{  g(x,y)}\,\mu(dx) \,\Big]\, \rho(dy)
\ge \int_\Y\Big[ \, \int_\X g(x,y)\,\gamma(dx)   -  H(\gamma\,\vert\,\mu) \,\Big]\, \rho(dy)\\
&\quad =  \int_\X  \Big[  \int_\Y g(x,y)\, \rho(dy) \Big] \,\gamma(dx)   -  H(\gamma\,\vert\,\mu)
=   \log \int_\X e^{\int_\Y g(x,y)\,\rho(dy)}\,\mu(dx) 
\end{align*}
where the last equality comes from taking 
\[  \gamma(dx)=\Bigl(\int_\X e^{\int_\Y g(z,y)\,\rho(dy)}\,\mu(dz) \Bigr)^{-1}  
e^{\int_\Y g(x,y)\,\rho(dy)}\,\mu(dx) .  \qedhere \]
\end{proof} 

\begin{lemma}\label{lm:lsc}
Let $\mathbb S$ and $\mathbb T$ be compact metric spaces and   $\pi:\mathbb S\to \mathbb T$ continuous.
Let $f:\mathbb S\to[0,\infty]$ be an arbitrary function and
$f_{\rm lsc}(s)=\lim_{r\searrow 0}\inf_{x\in B(s,r)} f(x)$  its lower semicontinuous regularization.    
Let $g(t)=\inf_{\pi(s)=t} f(s)$.   Then
$g_{\rm lsc}(t)= \inf_{\pi(s)=t} f_{\rm lsc}(s)$.
\end{lemma}

\begin{proof}  Immediately  $g_{\rm lsc}(t)\ge \inf_{\pi(s)=t} f_{\rm lsc}(s)$
because the function on the right is at or below $g(t)$ and on a compact metric space it is l.s.c.

Let $c >  \inf_{\pi(s)=t} f_{\rm lsc}(s)$.  Fix $s$ so that $\pi(s)=t$ and
$f_{\rm lsc}(s) < c$.  Find  $s_j \to s$ so that $f(s_j) < c$  (constant sequence
$s_j=s$ is a legitimate choice).  Then $\pi(s_j)\to t$,  and consequently
\[g_{\rm lsc}(t) \le \varliminf g(\pi(s_j))  \le  \varliminf f(s_j)  \le c.\qedhere\]
\end{proof}

%

\section{Proof of Lemma \ref{CLASS K}}
In what follows, $C$ denotes a chameleon 
constant which can change values from line to line. The only values it depends upon
are $|\rrange|$, $\ell$, and $d$. $C_r$ is again a chameleon constant but its value also
depends on $r$. Finally, $C_r(\w)$ also depends on $\w$. Note that $\ell\ge1$ is a fixed
integer, throughout this section.

Recall that $\xtil_\ell=\ztil_1+\dotsc+\ztil_\ell$. Similarly, $\xbar_\ell=\zbar_1+\dotsc+\zbar_\ell$.
Under \eqref{RWRE-span} there always exists
a path from $(y,\tilde z_{1,\ell})$ to $(x,z_{1,\ell})$ in the sense that
there exist $m\ge \ell$ and $a_1,\dotsc,a_{m-\ell}\in\rrange$ such that 
	\[  y+\xtil_\ell+ a_1+\dotsm+a_{m-\ell}=x. \]
The definition is independent of $z_{1,\ell}$ but for symmetry of language it seems
sensible to keep it in the statement. The case $m=\ell$ is admissible also and then
$y+\xtil_\ell=x$.   Then if we set $a_{m-\ell+1, m}=z_{1,\ell}$,
the composition $\Sopr_{a_m}\circ\dotsm\circ\Sopr_{a_1}$ takes 
$(T_y\w, \ztil_{1,\ell})$ to $(T_x\w,z_{1,\ell})$ for all $\w\in\Omega$.  

Paths can be concatenated.  If there is a path from $(y,\ztil_{1,\ell})$ to $(x,z_{1,\ell})$
and from $(u, \zbar_{1,\ell})$ to $(y,\ztil_{1,\ell})$, then we have 
	\[  y+\xtil_\ell+ a_1+\dotsm+a_{m-\ell}=x 
		\quad\text{and}\quad u+\xbar_\ell+ b_1+\dotsm+b_{n-\ell}=y. \]
Taking $b_{n-\ell+1,n}=\tilde z_{1,\ell}$ we then have 
	\[   u+\xbar_\ell+ b_1+\dotsm+b_{n} + a_1+\dotsm+a_{m-\ell}=x \]
and there is a path from $(u, \zbar_{1,\ell})$ to $(x,z_{1,\ell})$. 

For any two points $(x,z_{1,\ell})$ and $(\xbar,\zbar_{1,\ell})$ and any $\ztil_{1,\ell}$
there exists a point $y\in\Z^d$ such that from $(y,\ztil_{1,\ell})$  there is a path to  both $(x,z_{1,\ell})$ 
and $(\xbar,\zbar_{1,\ell})$.  For this, find  first 
$\abar_1,\dotsc,\abar_{m-\ell}$ and $a_1,\dotsc,a_{n-\ell}\in\rrange$ such that 
	\[  \xbar- x = (\abar_1+\dotsm +\abar_{m-\ell})-(  a_1+\dotsm+a_{n-\ell}) \]
so that 
	\[y'=\xbar-  (\abar_1+\dotsm +\abar_{m-\ell}) = x  -(  a_1+\dotsm+a_{n-\ell})\]
and then take $y= y'-\xtil_\ell$.   By induction, there is a common starting point for
paths to any finite number of points.  

Now fix $F\in\K^p(\bigom_\ell\times\range)$.
If there is a path from $(y,\ztil_{1,\ell})$ to $(x,z_{1,\ell})$, set $\wz_0=(T_y\w, \ztil_{1,\ell})$, 
$\wz_i=\Sopr_{a_i}\wz_{i-1}$ for $i=1,\dotsc, m$ so that $\wz_m=(T_x\w, z_{1,\ell})$,
and then   
	\begin{align}
  		L(\w, (y,\ztil_{1,\ell}), (x,z_{1,\ell}))= \sum_{i=0}^{m-1} F(\wz_i, a_{i+1}).  \label{defL}
	\end{align}
By the closed loop property $L(\w, (y,\ztil_{1,\ell}), (x,z_{1,\ell}))$ is independent of the
path chosen. 
If $a_1,\dotsc,a_{m-\ell}$  work for  $(y,\ztil_{1,\ell})$ and  $(x,z_{1,\ell})$,
then these steps work also for  $(y+u,\ztil_{1,\ell})$ and  $(x+u,z_{1,\ell})$. 
The effect on the right-hand side of \eqref{defL} is simply to shift $\w$ by $u$, and
consequently 
	\begin{align}  
	L(T_u\w, (y,\ztil_{1,\ell}), (x,z_{1,\ell}))=L(\w, (y+u,\ztil_{1,\ell}), (x+u,z_{1,\ell})). 
	\label{shiftL}
	\end{align}

Next define $f:\Omega\times\rrange^{2\ell}\times\Z^d\to\R$ by
	\begin{align}   
		f(\w, z_{1,\ell}, \zbar_{1,\ell}, x)= L(\w, (y,\ztil_{1,\ell}), (x,\zbar_{1,\ell}))
			- L(\w, (y,\ztil_{1,\ell}), (0,z_{1,\ell}))  \label{deff}
	\end{align}
for any $(y,\ztil_{1,\ell})$ with  a path to  both $(0,z_{1,\ell})$ and $(x,\zbar_{1,\ell})$.   
This definition is  independent of the choice of  $(y,\ztil_{1,\ell})$, again by the closed loop property.

Here are some basic properties of $f$. We postpone the proof of this lemma to the end of this section.

\begin{lemma}\label{f properties} Same setting as Lemma \ref{CLASS K}. 
\begin{itemize}
\item[{\rm (a)}]  There exists a constant $C$ depending only on $d$, $\ell$, and $R=\max\{|z|:z\in\rrange\}$, such that we have for all
$z_{1,\ell},\zbar_{1,\ell}\in\rrange^\ell$, $x\in\Z^d$, and $\P$-a.e.\ $\w$,
 	\[|f(\w,z_{1,\ell},\zbar_{1,\ell},x)|\le\max_{\substack{\ztil_{1,\ell}\in\rrange^\ell\\z\in\rrange}}\sum_{b:|b|\le C|x|}|F(T_b\w,\ztil_{1,\ell},z)|.\]
\item[{\rm (b)}]  The closed loop property of $F$ implies that for any $z_{1,\ell},\zbar_{1,\ell}\in\rrange^\ell$, $x,\xbar\in\Z^d$, and $\P$-a.e.\ $\w$, 
 \begin{align*} 
f(T_x\w, \zbar_{1,\ell}, \zbar_{1,\ell}, \xbar-x)
&=f(\w, \zbar_{1,\ell}, \zbar_{1,\ell}, \xbar) -  f(\w, \zbar_{1,\ell}, \zbar_{1,\ell}, x)\\
&=f(\w, z_{1,\ell}, \zbar_{1,\ell}, \xbar) -  f(\w, z_{1,\ell}, \zbar_{1,\ell}, x).
\end{align*}
\item[{\rm (c)}]  The mean zero property of $F$ implies that for any $\zbar_{1,\ell}\in\rrange^\ell$ and $x\in\Z^d$, $\E[f(\w, \zbar_{1,\ell},  \zbar_{1,\ell}, x)]=0.$
\end{itemize}
\end{lemma}

Next, extend $f$ to a continuous function of $\xi\in\R^d$ by linear interpolation.
Here is one way to do that. Recall that $\{e_1,\dotsc,e_d\}$ is the canonical basis of $\R^d$.
Introduce the following notation: for $p\in[0,1]$ and $i\in\{1,\dotsc,d\}$, let $\Ber_i(p)$ be
a Bernoulli random variable with parameter $p$. For a vector  
$p=(p_1,\dotsc,p_d)\in[0,1]^d$, let
$\Ber(p)=\sum_{i=1}^d \Ber_i(p_i) e_i$ with $(\Ber_i(p_i))$ independent.

Now, for given $\wz$, $\bar z_{1,\ell}$, and 
$\xi=\sum_{i=1}^d \xi_i e_i$, let $[\xi]=\sum_{i=1}^d[\xi_i] e_i$, 
where $[\xi_i]$ is the largest integer smaller than or equal to $\xi_i$, and define
	\[f(\wz,\bar z_{1,\ell},\xi)=E[f(\wz,\bar z_{1,\ell},[\xi]+\Ber(\xi-[\xi]))].\]

Think of $f$ as a collection of
functions of $(\w,\xi)$. The idea is to homogenize these
functions by showing that, for fixed $z_{1,\ell}$ and $\bar z_{1,\ell}$ and for $\P$-a.e.\ $\w$,
$g_n(\w,z_{1,\ell},\bar z_{1,\ell},\xi)=n^{-1}f((\w,z_{1,\ell}),\bar z_{1,\ell},n\xi)$ is equicontinuous and hence 
converges, uniformly on compacts and along a subsequence, to a function
$g(\w,z_{1,\ell},\bar z_{1,\ell},\xi)$.
Next, one shows that $g$ has to be constant and since $g(\w,z_{1,\ell},\bar z_{1,\ell},0)=0$ 
we conclude that
$g_n$ converges uniformly on compacts to 0. Observe now that if $\wz_0=(\w,z_{1,\ell})$ and $\wz_{k+1}=\Sopr_{a_{k+1}}\wz_k$ for $0\le k\le n-1$ and $a_k\in\rrange$,
then
	\[n^{-1}\sum_{k=0}^{n-1} F(\wz_k,a_{k+1})=g_n(\w,z_{1,\ell},\zbar_{1,\ell},\xi),\]
where $\xi=(x_\ell+a_1+\cdots+a_{n-\ell})/n$ and $\zbar_{1,\ell}=(a_{n-\ell+1},\dotsc,a_n)$. Thus,
\begin{align*}
&\max_{(a_1,\cdots,a_n)\in\rrange^n}
\Big|n^{-1}\sum_{k=0}^{n-1} F(\wz_k,a_{k+1})\Big|
\le\max_{\zbar_{1,\ell}\in\rrange^\ell}\sup_{\xi:|\xi|\le R} |g_n(\w,z_{1,\ell},\zbar_{1,\ell},\xi)|,
\end{align*}
where $R=\max\{|z|:z\in\rrange\}$. This completes the proof of the lemma.

The above strategy was introduced by  Kosygina, Rezakhanlou, and Varadhan \citep{kosy-reza-vara-06} in the context of
diffusions with random drift, then carried out by Rosenbluth \citep{rose-thesis} for random walk in random environment in the case $\ell=0$. 
Equicontinuity follows from an application of the Garsia-Rodemich-Rumsey theorem (see \citep{stro-vara})
which requires the moment assumption on $F$. The fact that $g$ is constant follows from
an application of the ergodic theorem along with the mean 0 property of $F$. We present
the proof, adapted to our setting, for the sake of completeness. 

Let us start with equicontinuity. This will be shown by breaking the space into two parts.
Each of the following two lemmas covers one part.
Let us denote $B_r(\xi)=\{\zeta\in\R^d:|\zeta-\xi|\le r\}$. 

\begin{lemma}
Same assumptions on $\P$ and $F$ as in Lemma \ref{CLASS K}. 
Then, for any $r\ge1$ and any
$\gamma\in(0,d+p)$,  one has that
for $\P$-a.e.\ $\w$ and all $z_{1,\ell},\bar z_{1,\ell}\in\rrange^{\ell}$
\[\varlimsup_{n\to\infty} 
\int_{B_r(0)}\int_{B_{2d/n}(\xi)\cap B_r(0)}
\frac{|g_n(\w,z_{1,\ell},\bar z_{1,\ell},\xi)-
g_n(\w,z_{1,\ell},\bar z_{1,\ell},\zeta)|^p}{|\xi-\zeta|^\gamma}\,d\zeta\, d\xi=0.\]
\end{lemma}

\begin{proof}
Changing variables, the above integral can be rewritten as
\begin{align}
&\frac1{n^p}\int_{B_r(0)}\int_{B_{2d/n}(\xi)\cap B_r(0)}
\frac{|f(n\xi)-f(n\zeta)|^p}{|\xi-\zeta|^\gamma}\,d\zeta\, d\xi\nn\\
&=
\frac1{n^{2d+p-\gamma}}\int_{B_{rn}(0)}\int_{B_{2d}(\xi)\cap B_{rn}(0)}
\frac{|f(\xi)-f(\zeta)|^p}{|\xi-\zeta|^\gamma}\,d\zeta\, d\xi,\label{integral1}
\end{align}
where we dropped $\w$, $z_{1,\ell}$, and $\bar z_{1,\ell}$ from the arguments of $f$ for the moment.

Observe next that if $\zeta$ is on the boundary of a $\Z^d$-cell, i.e.\ 
$\zeta_i\in\Z$ for some $i\in\{1,\dotsc,d\}$, then the fact that $\zeta_i+B_i(0)$ has the
same distribution as $\zeta_i-1+B_i(1)$ shows that
one can set $[\zeta_i]$ to be either $\zeta_i$ or $\zeta_i-1$ and the value of $f$
at $\zeta$ would not be affected. 

Therefore, if $\xi$ and $\zeta$ belong to the same $\Z^d$-cell, we can
assume that $[\xi_i]=[\zeta_i]=x$, the lower left corner of the cell. Abbreviate $p_i=\zeta_i-x_i$ and
$q_i=\xi_i-x_i$. Then
\begin{align*}
&|f(\zeta)-f(\xi)|=|E[f(x+\Ber(\zeta-x))-f(x+\Ber(\xi-x))]|\\
&=
\Bigg|\sum_{(b_i)\in\{0,1\}^d}
\Big[\prod_i p_i^{b_i}(1-p_i)^{1-b_i} -
\prod_i q_i^{b_i}(1-q_i)^{1-b_i}\Big]
f\Big(x+\sum_i b_i e_i\Big)\Bigg|\\
&=
\Bigg|\sum_{(b_i)\in\{0,1\}^d}
\Big[\prod_i p_i^{b_i}(1-p_i)^{1-b_i} -
\prod_i q_i^{b_i}(1-q_i)^{1-b_i}\Big]
\Big[f\Big(x+\sum_i b_i e_i\Big)-f(x)\Big]\Bigg|\\
&\le
\sum_{(b_i)\in\{0,1\}^d}
\Big|\prod_i p_i^{b_i}(1-p_i)^{1-b_i} -
\prod_i q_i^{b_i}(1-q_i)^{1-b_i}\Big|\cdot
\Big|f\Big(x+\sum_i b_i e_i\Big)-f(x)\Big|\\
&\le
C\, |\zeta-\xi|\sum_{(b_i)\in\{0,1\}^d}
\Big|f\Big(x+\sum_i b_i e_i\Big)-f(x)\Big|\,,
\end{align*}
where we have used the fact that for $a,b,c,d\in[0,1]$,
\[|ab-cd|\le |(a-c)b|+|(b-d)c|\le|a-c|+|b-d|.\]

If, on the other hand, $\xi$ and $\zeta$ are in two different $\Z^d$-cells then, 
since $|\xi-\zeta|\le 2d$,
there exist points $\zeta_0,\cdots,\zeta_m$, with $m\le C$, such that
$\zeta_0=\xi$, $\zeta_m=\zeta$,
each two consecutive ones belong to the same $\Z^d$-cell, and 
$|\zeta_{k+1}-\zeta_k|\le C\,|\zeta-\xi|$.
One can then write
\begin{align*}
|f(\zeta)-f(\xi)|&\le C\sum_{k=0}^{m-1}|\zeta_{k+1}-\zeta_k|
\sum_{(b_i)\in\{0,1\}^d}
\Big|f\Big([\zeta_k]+\sum_i b_i e_i\Big)-f([\zeta_k])\Big|\\
&\le C\,|\zeta-\xi|\sum_{k=0}^{m-1}
\sum_{(b_i)\in\{0,1\}^d}
\Big|f\Big([\zeta_k]+\sum_i b_i e_i\Big)-f([\zeta_k])\Big|\\
&= C\,|\zeta-\xi|\sum_{k=0}^{m-1}
\sum_{(b_i)\in\{0,1\}^d}
\Big|f\Big(T_{[\zeta_k]}\w,\zbar_{1,\ell},\zbar_{1,\ell},\sum_i b_i e_i\Big)\Big|,
\end{align*}
where we have used part (b) of Lemma \ref{f properties}.
Furthermore, using part (a) of the same lemma, and that $|\zeta_k-[\xi]|\le C\,|\xi-\zeta|\le C$, we have
\[|f(\zeta)-f(\xi)|\le C\,|\zeta-\xi|\max_{\tilde z_{1,\ell}\in\rrange^{\ell}}\max_{z\in\rrange}
\max_{x:|x-[\xi]|\le C}|F(T_x\w,\tilde z_{1,\ell},z)|.\]
Since $d+p>\gamma$, one has that $\int_{B_{2d}(\xi)}|\zeta-\xi|^{p-\gamma}d\zeta<\infty$.
Setting,
\begin{align}
\label{G}
G(\w)=\max_{\tilde z_{1,\ell}\in\rrange^{\ell}}\max_{z\in\rrange}
\max_{x:|x|\le C}|F(T_x\w,\tilde z_{1,\ell},z)|^p\in L^1(\P),
\end{align}
integral \eqref{integral1} is then bounded by
\[C\,n^{\gamma-d-p}\Big(n^{-d}\sum_{y:|y|\le rn}G(T_y\w)\Big).\]
The lemma follows since $d+p>\gamma$ and, by the ergodic theorem
(see for example Theorem 14.A8 in \citep{geor}), the quantity 
in parentheses converges to a finite constant.
\end{proof}

\begin{lemma}
Same assumptions on $\P$ and $F$ as in Lemma \ref{CLASS K}. 
Then, for any $r\ge1$ and any $\gamma\in(d+p-1,d+p)$,
there exists a constant $C_r$ such that 
for $\P$-a.e.\ $\w$ and all $z_{1,\ell},\bar z_{1,\ell}\in\rrange^{\ell}$
\[\varlimsup_{n\to\infty} 
\int_{B_r(0)}\int_{B_r(0)\smallsetminus B_{2d/n}(\xi)}
\frac{|g_n(\w,z_{1,\ell},\bar z_{1,\ell},\xi)-
g_n(\w,z_{1,\ell},\bar z_{1,\ell},\zeta)|^p}{|\xi-\zeta|^\gamma}\,d\zeta\, d\xi\le C_r\,.\]
\end{lemma}

\begin{proof}
Once again, changing variables the above integral becomes
\begin{align}
\frac1{n^{2d+p-\gamma}}\int_{B_{rn}(0)}\int_{B_{rn}(0)\smallsetminus B_{2d}(\xi)}
\frac{|f(\xi)-f(\zeta)|^p}{|\xi-\zeta|^\gamma}\,d\zeta\, d\xi.\label{integral2}
\end{align}

Write
\begin{align*}
|f(\xi)-f(\zeta)|\le \ &|f([\xi])-f([\zeta])|\\
&+\sum_{(b_i)\in\{0,1\}^{d}}\Big|f\Big([\xi]+\sum_i b_i e_i\Big)-f([\xi])\Big|
+\sum_{(b_i)\in\{0,1\}^{d}}\Big|f\Big([\zeta]+\sum_i b_i e_i\Big)-f([\zeta])\Big|.
\end{align*}
Observing that $\gamma>1$, $\gamma<d+p$, $|\xi-\zeta|\ge 2d$, and
$|[\zeta]|\le C|\zeta|\le C_r n$, the second and third terms above are 
dealt with exactly as in the previous lemma (using the ergodic theorem). 
For example, 
	\begin{align*}
		&\frac1{n^{2d+p-\gamma}}\int_{B_{rn}(0)}\int_{B_{rn}(0)\smallsetminus B_{2d}(\xi)}
		\frac{\Big|f\Big([\xi]+\sum_i b_i e_i\Big)-f([\xi])\Big|^p}{|\xi-\zeta|^\gamma}\,d\zeta\, d\xi\\
		&\qquad\le\frac C{n^{2d+p-\gamma}}\int_{B_{rn}(0)}\Big|f\Big([\xi]+\sum_i b_i e_i\Big)-f([\xi])\Big|^p\,d\xi\\
		&\qquad\le C{n^{\gamma-d-p}}\Big(n^{-d}\sum_{y:|y|\le rn}G(T_y\w)\Big).
	\end{align*}

Observe next that since $|\xi-\zeta|\ge2d$, $1\le|[\xi]-[\zeta]|\le C\, |\xi-\zeta|$ and we are reduced to bounding
the sum
\begin{align*}
\frac1{n^{2d+p-\gamma}}\sum_{\substack{x,y:x\ne y\\ |x|,|y|\le rn}}
\frac{|f(x)-f(y)|^p}{|x-y|^\gamma}.
\end{align*}
Now, $|f(x)-f(y)|^p\le C\,m^{p-1}\sum_{i=1}^m G(T_{x_i}\w)$, where $G$ was defined in \eqref{G} and
$(x_i)$ is any path in $\Z^d$
from $x$ to $y$, with length $m\le C\,|x-y|$. If one chooses canonical paths
that go from each $x$ to each $y$ and that stay as close as possible to the line connecting
$x$ and $y$, e.g.\ staying at distance less than $d$ from the line,
then the above sum is bounded by
\begin{align*}
\frac1{n^{2d+p-\gamma}}\sum_{s:|s|\le C_r n} A_{s,n,\gamma} 
G(T_s\w),
\end{align*}
where 
\[A_{s,n,\gamma}=\sum_{\substack{x,y:x\ne y\\ |x|,|y|\le rn}}
|x-y|^{p-1-\gamma}\one\{s\text{ is on the canonical path from $x$ to $y$}\}.\]
Consider a fixed $s$. For a given integer $\rho_1$, there are at most $C_r\rho_1^{d-1}$ 
$x$'s such that $|x-s|=\rho_1$. Fix such an $x$. See Figure \ref{y-count}.
\begin{figure}[tb]
\includegraphics[width=0.4\textwidth]{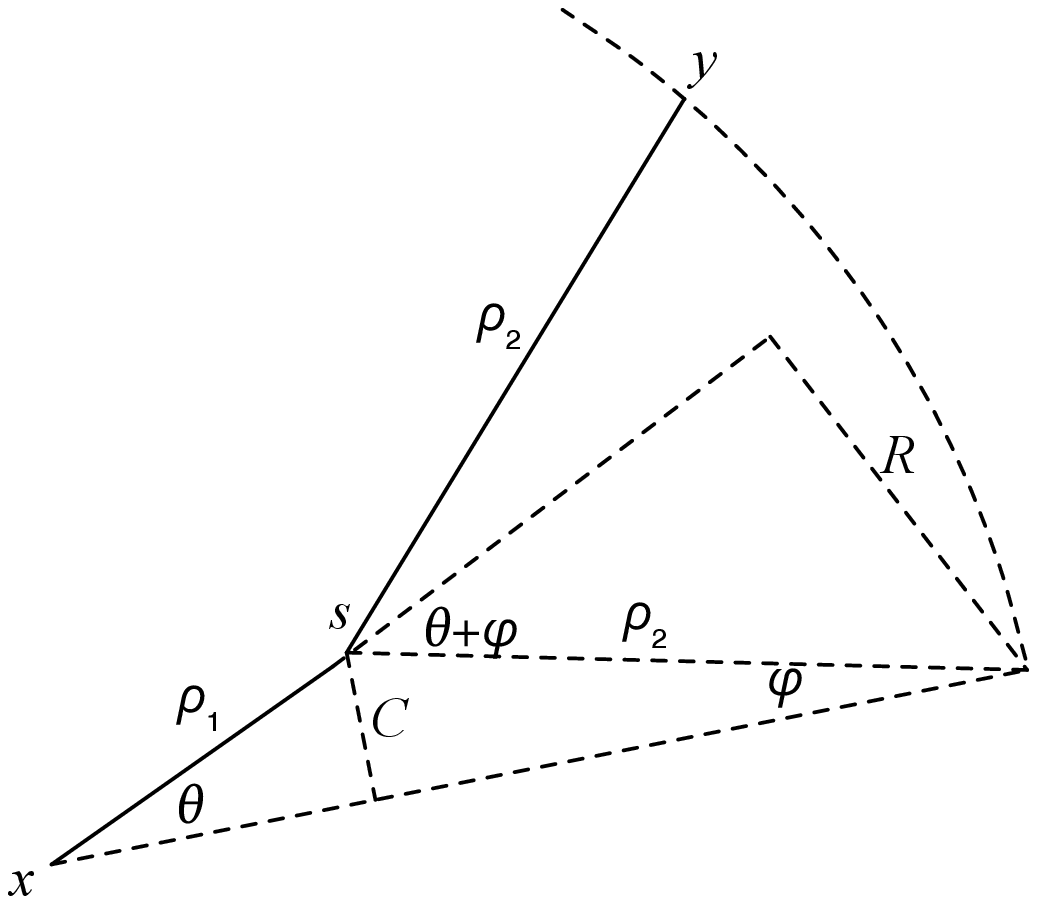}
\caption[]{$y$ count.}
\label{y-count}
\end{figure}
Because the line joining $x$ and $y$ has to be within
a bounded distance of $s$, radius $R$ is bounded by
	\[R\le \rho_2\sin(\theta+\varphi)\le \rho_2(\sin\theta+\sin\varphi)\le C(1+\rho_2/\rho_1).\]
Hence, there can  be at most $C_r(1+\rho_2/\rho_1)^{d-1}$
possible $y$'s with $|y-s|=\rho_2$ being a given integer. 
Thus, there are at most $C_r(\rho_1+\rho_2)^{d-1}$ pairs $(x,y)$ that have $s$
on the canonical path joining them. Furthermore, $\rho_1+\rho_2\le C_r|x-y|$.
Therefore, 
\[A_{s,n,\gamma}\le C_r\sum_{\rho_1,\rho_2=1}^{C_r n} (\rho_1+\rho_2)^{d+p-2-\gamma}\le 
C_r n^{d+p-\gamma}.\]
This allows us to bound the above sum by
\[C_r\, n^{-d}\sum_{s\in\Z^d:|s|\le C_r n} G(T_s\w),\]
which, by the ergodic theorem, converges to a constant.
\end{proof}

We have shown that for a fixed $r\ge1$, 
if $d+p-1<\gamma<d+p$, then for all $z_{1,\ell}$ and $\bar z_{1,\ell}$ and 
$\P$-a.e.\ $\w$
\[\sup_n 
\int_{B_r(0)}\int_{B_r(0)}
\frac{|g_n(\w,z_{1,\ell},\bar z_{1,\ell},\xi)-
g_n(\w,z_{1,\ell},\bar z_{1,\ell},\zeta)|}{|\xi-\zeta|^\gamma}\,d\zeta\, d\xi\le C_r(\w).\]
Next, we apply an extension of Theorem 2.1.3 in \citep{stro-vara}; see Exercise 2.4.1 therein.
\newtheorem*{GRR}{Garsia-Rodemich-Rumsey's Theorem}
\begin{GRR}
Let $g:\R^d\to\R$ be a continuous function on $B_r(0)$ for some $r>0$. Let $\gamma>0$. 
If
	\[\int_{B_r(0)}\int_{B_r(0)}\frac{|g(\xi)-g(\zeta)|}{|\xi-\zeta|^\gamma}\,d\zeta\,d\xi\le C_r,\]
then for $\xi,\zeta\in B_{r/2}(0)$,
	\[|g(\xi)-g(\zeta)|\le C'_r|\xi-\zeta|^{\gamma-2d},\]
where $C'_r$ depends on $C_r$ and the dimension $d$.
\end{GRR}
From this theorem it follows that
\[\sup_n 
|g_n(\w,z_{1,\ell},\bar z_{1,\ell},\xi)-g_n(\w,z_{1,\ell},\bar z_{1,\ell},\zeta)|
\le C_r(\w)\,|\xi-\zeta|^{\gamma-2d}.\]
Since $2d<d+p$, there exists a suitable $\gamma$ such that $\gamma-2d>0$. 
This shows that $\{g_n(\w,z_{1,\ell},\bar z_{1,\ell},\xi)\}$ is equicontinuous in $\xi\in B_r(0)$, for 
all $r\ge1$.
Let $g(\w,z_{1,\ell},\bar z_{1,\ell},\xi)$ be a uniform (on compacts) limit point, for fixed $\w$. 

Now compute, for any fixed $i_0\in\{1,\dotsc,d\}$ and $\xi=\sum_{i=1}^d\xi_i e_i$ with $\xi_i\ge0$,
\begin{align}
&\Bigg|n^{-(d-1)}\!\!\!\!\sum_{\substack{0\le k_i< [n\xi_i]\\ 1\le i\le d, i\ne i_0}}\!\!\!\!
g_n\Big(\tfrac{[n\xi_{i_0}]}n e_{i_0}+\sum_{j\ne i_0}\tfrac {k_j}n e_j\Big)
-\int\prod_{i\ne i_0}\one\{0\le\zeta_i\le\xi_i\} 
g\Big(\xi_{i_0}e_{i_0}+\sum_{j\ne i_0}\zeta_j e_j\Big)\prod_{i\ne i_0}d\zeta_i\Bigg|\nn\\
&\le
\Bigg|n^{-(d-1)}\!\!\!\!\sum_{\substack{0\le k_i< [n\xi_i]\\ 1\le i\le d, i\ne i_0}}\!\!\!\!
g_n\Big(\tfrac{[n\xi_{i_0}]}n e_{i_0}+\sum_{j\ne i_0}\tfrac {k_j}n e_j\Big)
-
n^{-(d-1)}\!\!\!\!\sum_{\substack{0\le k_i< [n\xi_i]\\ 1\le i\le d, i\ne i_0}}\!\!\!\!
g\Big(\tfrac{[n\xi_{i_0}]}n e_{i_0}+\sum_{j\ne i_0}\tfrac {k_j}n e_j\Big)
\Bigg|\label{line-1}\\
&+
\Bigg|n^{-(d-1)}\!\!\!\!\sum_{\substack{0\le k_i< [n\xi_i]\\ 1\le i\le d, i\ne i_0}}\!\!\!\!
g\Big(\tfrac{[n\xi_{i_0}]}n e_{i_0}+\sum_{j\ne i_0}\tfrac {k_j}n e_j\Big)
-\int\prod_{i\ne i_0}\one\{0\le\zeta_i\le\xi_i\} 
g\Big(\xi_{i_0}e_{i_0}+\sum_{j\ne i_0}\zeta_j e_j\Big)\prod_{i\ne i_0}d\zeta_i\Bigg|.
\label{line-2}
\end{align}
The term on line \eqref{line-1} converges to 0 because of the uniform convergence of 
$g_n$ to $g$ and the term on line \eqref{line-2} converges to 0 because $g$ is continuous
and the sum is a Riemann sum. Similarly,
\begin{align*}
\Bigg|n^{-(d-1)}\sum_{\substack{0\le k_i< [n\xi_i]\\ 1\le i\le d, i\ne i_0}}
g_n\Big(\sum_{j\ne i_0}\tfrac {k_j}n e_j\Big)-
\int\prod_{i\ne i_0}\one\{0\le\zeta_i\le\xi_i\} 
g\Big(\sum_{j\ne i_0}\zeta_j e_j\Big)\prod_{i\ne i_0}d\zeta_i\Bigg|
\end{align*}
converges to 0, as $n\to\infty$.

On the other hand, 
\begin{align*}
&n^{-(d-1)}\sum_{\substack{0\le k_i< [n\xi_i]\\ 1\le i\le d, i\ne i_0}}
g_n\Big(\tfrac{[n\xi_{i_0}]}n e_{i_0}+\sum_{j\ne i_0}\tfrac {k_j}n e_j\Big)
-
n^{-(d-1)}\sum_{\substack{0\le k_i< [n\xi_i]\\ 1\le i\le d, i\ne i_0}}
g_n\Big(\sum_{j\ne i_0}\tfrac {k_j}n e_j\Big)\\
&\qquad=
n^{-d}\sum_{\substack{0\le k_i< [n\xi_i]\\ 1\le i\le d, i\ne i_0}}
f\Big([n\xi_{i_0}] e_{i_0}+\sum_{j\ne i_0} k_j e_j\Big)
-
n^{-d}\sum_{\substack{0\le k_i< [n\xi_i]\\ 1\le i\le d, i\ne i_0}}
f\Big(\sum_{j\ne i_0} k_j e_j\Big)\\
&\qquad=
n^{-d}\sum_{\substack{0\le k_i< [n\xi_i]\\ 1\le i\le d}}
\Big\{f\Big(e_{i_0}+\sum_j k_j e_j\Big)-f\Big(\sum_j k_j e_j\Big)\Big\}\\
&\qquad=
n^{-d}\sum_{x\in V_n} G'(T_x\w),
\end{align*}
where $G'(\w)=f(\w,\bar z_{1,\ell},\bar z_{1,\ell},e_{i_0})\in L^1(\P)$  and
\[V_n=\Big\{x=\sum_{i=1}^d k_i e_i:0\le k_i<[n\xi_i],\forall i\Big\}.\]
For the last equality above we used (b) of Lemma \ref{f properties}.
By (c) of Lemma \ref{f properties} we have 
$\E[G']=0$ 
and the ergodic theorem
implies that the above converges to 0.

We have thus shown that 
\begin{align*}
&\int\prod_{i\ne i_0}\one\{0\le\zeta_i\le\xi_i\} 
g\Big(\xi_{i_0}e_{i_0}+\sum_{j\ne i_0}\zeta_j e_j\Big)\prod_{i\ne i_0}d\zeta_i
=\int\prod_{i\ne i_0}\one\{0\le\zeta_i\le\xi_i\} 
g\Big(\sum_{j\ne i_0}\zeta_j e_j\Big)\prod_{i\ne i_0}d\zeta_i\,
\end{align*}
which implies that 
\begin{align*}
&\!\!\!\!\int\prod_{i\ne i_0}\one\{\xi_i'\le\zeta_i\le\xi_i\} 
g\Big(\xi_{i_0}e_{i_0}+\sum_{j\ne i_0}\zeta_j e_j\Big)\prod_{i\ne i_0}d\zeta_i
=\int\prod_{i\ne i_0}\one\{\xi_i'\le\zeta_i\le\xi_i\} 
g\Big(\xi_{i_0}' e_{i_0}+\sum_{j\ne i_0}\zeta_j e_j\Big)\prod_{i\ne i_0}d\zeta_i\,
\end{align*}
and hence $g$ is independent of $\xi_{i_0}$, for all $i_0\in\{1,\dotsc,d\}$.
This means $g(\w,z_{1,\ell},\bar z_{1,\ell},\xi)=g(\w,z_{1,\ell},\bar z_{1,\ell},0)=0$. In other words,
$g_n$ converges uniformly (on compacts) to $g=0$. Lemma \ref{CLASS K} is thus proved.

\begin{proof}[Proof of Lemma \ref{f properties}]  
Recall that $\{e_1,\dotsc,e_d\}$ be the canonical basis of $\R^d$. 
For each $1\le i\le d$, there exist $n_i$, $m_i$, $(a_{i,j})_{j=1}^{n_i}$, and $(\abar_{i,j})_{j=1}^{m_i}$ from $\rrange$ such that
	\[e_i=\abar_{i,1}+\cdots+\abar_{i,m_i}-a_{i,1}-\cdots-a_{i,n_i}.\]
Write $x=\sum_{i=1}^d b_i e_i$. Then, 
	\[x=\sum_{i=1}^d\sum_{j=1}^{m_i}b_i\abar_{i,j}-\sum_{i=1}^d\sum_{j=1}^{n_i}b_i a_{i,j}.\]
One can thus find a $y$ that has a path to both $0$ and $x$ and such that $|y|\le C|x|$. This proves (a).
  
To prove (b),  let  $(y,\ztil_{1,\ell})$ have a path to  both $(x,\zbar_{1,\ell})$ 
and $(\xbar,\zbar_{1,\ell})$.  Find $(y',\tilde z'_{1,\ell})$ that has a path to both $(y,\tilde z_{1,\ell})$ and 
$(0,z_{1,\ell})$.  Then, from \eqref{deff},
\begin{align*}
 &f(\w, z_{1,\ell}, \bar z_{1,\ell}, \bar x) -  f(\w, z_{1,\ell}, \bar z_{1,\ell}, x)\\
&\quad = \Bigl[  L(\w, (y',\tilde z'_{1,\ell}), (y,\tilde z_{1,\ell}))
+ L(\w, (y,\tilde z_{1,\ell}), (\bar x,\bar z_{1,\ell})) 
 - L(\w, (y',\tilde z'_{1,\ell}), (0,z_{1,\ell})) \Bigr]  \\
&\qquad -\;  \Bigl[  L(\w, (y',\tilde z'_{1,\ell}), (y,\tilde z_{1,\ell}))
+ L(\w, (y,\tilde z_{1,\ell}), (x,\bar z_{1,\ell})) 
 - L(\w, (y',\tilde z'_{1,\ell}), (0,z_{1,\ell})) \Bigr] \\
&\quad  = L(\w, (y,\tilde z_{1,\ell}), (\bar x,\bar z_{1,\ell}))
- L(\w, (y,\tilde z_{1,\ell}), (x,\bar z_{1,\ell})).
\end{align*}
The last line above is independent of $z_{1,\ell}$ so we can substitute 
$\bar z_{1,\ell}$ for $z_{1,\ell}$ and get the second equality of part (b). 
For the first equality, by the definition of $f$ \eqref{deff}, the shift property \eqref{shiftL}, and the second equality in (b) just proved, 
we have for a new $(y,\ztil_{1,\ell})$
\begin{align*}
& f(T_x\w, \bar z_{1,\ell}, \bar z_{1,\ell}, \bar x-x)\\
&\qquad=L(T_x\w, (y,\tilde z_{1,\ell}), (\bar x-x,\bar z_{1,\ell}))
- L(T_x\w, (y,\tilde z_{1,\ell}), (0,\bar z_{1,\ell})) \\
&\qquad =  L(\w, (y+x,\tilde z_{1,\ell}), (\bar x,\bar z_{1,\ell}))
- L(\w, (y+x,\tilde z_{1,\ell}), (x,\bar z_{1,\ell})) \\
&\qquad =  f(\w, \bar z_{1,\ell}, \bar z_{1,\ell}, \bar x)
-  f(\w, \bar z_{1,\ell},  \bar z_{1,\ell}, x). 
\end{align*}

For (c), by the earlier observation, we can choose $y$ so that 
from $(y,\bar z_{1,\ell})$  there is a path to  both $(x,\bar z_{1,\ell})$ 
and $(0,\bar z_{1,\ell})$.  Then 
\[   f(\w, \bar z_{1,\ell}, \bar z_{1,\ell}, x)= L(\w, (y,\bar z_{1,\ell}), (x,\bar z_{1,\ell}))
- L(\w, (y,\bar z_{1,\ell}), (0, \bar z_{1,\ell})).  \]
Both $L$-terms above equal sums 
$\sum_{i=0}^{m-1} F(\wz_i, a_{i+1})$ where $\wz_0=(T_y\w, \bar z_{1,\ell})$
and $\wz_m=(T_u\w, \bar z_{1,\ell})$ with $u=x$ or $u=0$.  Both have 
$\E$-mean zero by property (ii) of Definition \ref{cK-def}.  
\end{proof}

\bibliographystyle{acmtrans-ims}
\bibliography{tmfrefs}
\end{document}